\documentclass{article}

\usepackage[utf8]{inputenc}
\usepackage[cyr]{aeguill}
\usepackage{graphicx}
\usepackage{bm}
\usepackage[a4paper]{geometry}
\usepackage{amsmath,amsfonts,amsthm,amssymb}
\usepackage{stmaryrd}
\usepackage{color,xcolor}
\usepackage{subcaption}
\usepackage{array}
\usepackage{multirow}
\usepackage{url}
\usepackage{hyperref}
\usepackage{centernot}
\usepackage{tikz}
\usepackage[export]{adjustbox}

\newtheorem{example}{Example}
\newtheorem{remark}{Remark}
\newtheorem{proposition}{Proposition}

\newtheorem{definition}{Definition}
\newcommand{\R}{\mathbb R}

\renewcommand{\P}{\mathbb P}
\newcommand{\M}{\mathbb M}
\newcommand{\bbf}{\mathbf{f}}
\newcommand{\bbg}{\mathbf{g}}
\newcommand{\bbu}{\mathbf{u}}
\newcommand{\bbv}{\mathbf{v}}

\newcommand{\bF}{\mathbf{F}}
\newcommand{\bA}{\mathbf{A}}
\newcommand{\dpar}[2]{\dfrac{\partial #1}{\partial #2}}

\newcommand{\momentsmatrix}{\mathbf{Q}}

\DeclareSymbolFont{matha}{OML}{txmi}{m}{it}
\DeclareMathSymbol{\varv}{\mathord}{matha}{118}
\begin{document}
\title{A new local and explicit kinetic method for linear and non-linear convection-diffusion problems with finite kinetic speeds: \\ II. Multi-dimensional case}
\author{Gauthier Wissocq\footnote{corresponding author}, R\'emi Abgrall\\
Institute of Mathematics, University of Z\"urich, Switzerland\\
gauthier.wissocq@math.uzh.ch, remi.abgrall@math.uzh.ch}
\date{}
\maketitle
\begin{abstract}
We extend to multi-dimensions the work of~\cite{wissocq2023Kinetic}, where new fully explicit kinetic methods were built for the approximation of linear and non-linear convection-diffusion problems. The fundamental principles from the earlier work are retained: (1) rather than aiming for the desired equations in the strict limit of a vanishing relaxation parameter, as is commonly done in the diffusion limit of kinetic methods, diffusion terms are sought as a first-order correction of this limit in a Chapman-Enskog expansion, (2) introducing a coupling between the conserved variables within the relaxation process by a specifically designed collision matrix makes it possible to systematically match a desired diffusion. Extending this strategy to multi-dimensions cannot, however, be achieved through simple directional splitting, as diffusion is likely to couple space directions with each other, such as with shear viscosity in the Navier-Stokes equations. In this work, we show how rewriting the collision matrix in terms of moments can address this issue, regardless of the number of kinetic waves, while ensuring conservation systematically. This rewriting allows for introducing a new class of kinetic models called \emph{regularized} models, simplifying the numerical methods and establishing connections with Jin-Xin models. Subsequently, new explicit arbitrary high-order kinetic schemes are formulated and validated on standard two-dimensional cases from the literature. Excellent results are obtained in the simulation of a shock-boundary layer interaction, validating their ability to approximate the Navier-Stokes equations with kinetic speeds obeying nothing but a subcharacteristic condition along with a hyperbolic constraint on the time step.
\end{abstract}


\section{Introduction}

We are interested in the approximation of linear and non-linear systems of convection-diffusion equations using kinetic methods. In the well-documented context of hyperbolic conservation systems, thus in the absence of diffusion, the principle of kinetic methods initially proposed by Jin and Xin~\cite{Jin} is to supplement the target equations by the evolution of new flux variables, assumed to be in the form of advection-relaxation. This can be done in such a way that, in the limit of a vanishing relaxation parameter, the flux variables converge toward the expected flux, so that the solution of the kinetic model formally converges toward that of the target hyperbolic system. Interestingly, diagonalizing the hyperbolic part of the Jin-Xin model makes it possible to rewrite it as an advection-relaxation system at constant kinetic speeds, acting on functions that are referred to as distribution functions by analogy with the Boltzmann equation in the kinetic theory of gases~\cite{Natalini, AregbaNatalini}. In particular, relaxation can be expressed as a Bhatnagar-Gross-Krook (BGK) collision~\cite{Bhatnagar1954}, where distribution functions are relaxed towards an equilibrium state referred to as Maxwellian. Eventually, this description can be extended to any number of kinetic speeds in any spatial direction, introducing new free parameters in the kinetic model~\cite{Natalini}. As demonstrated by Bouchut~\cite{Bouchut}, when the Maxwellian state obeys some monotonicity property, which, in many cases, can be reduced to Whitham subcharacteristic condition~\cite{Whitham1974, Jin}, then the kinetic model becomes compatible with entropy inequalities. This fundamental property, together with the fact that advection is performed at \emph{constant} kinetic speeds, has made it possible to build very robust numerical methods for the approximation of hyperbolic systems in extreme conditions~\cite{Jin, Caflisch1997, Natalini, AregbaNatalini, Pareschi2005, Boscarino2009, Boscarino2014, Bouchut2018, Cho2021, AbgrallK, Abgrall2023}.

Adapting this kinetic framework to the approximation of convection-diffusion systems is, however, a tedious task which has been the topic of many research studies over the last decades~\cite{Jin1998, Klar1998, Jin2000, Naldi2000, Jin2001, Aregba-Driollet2003, Lemou2008, LAFITTE20171, Boscarino2013, Jang2014, Peng2021}. The main difficulty arises from the fact that, in the diffusion limit, the kinetic speeds of Jin-Xin like schemes scale as $1/\varepsilon$, where $\varepsilon$ is the relaxation parameter. As a consequence, the use of explicit schemes to approximate the hyperbolic terms of kinetic models leads to a numerical stability constraint $\Delta t = \mathcal{O}(\varepsilon \Delta x)$, which is, in the diffusive limit where $\varepsilon < \Delta x$, more restrictive than the common parabolic constraint $\Delta t = \mathcal{O}(\Delta x^2)$~\cite{Jin1998, Boscarino2013, Peng2020}. Some possibilities can be considered to circumvent this issue, \textit{e.g.}, as reported in~\cite{Boscarino2013}, the specific design of so-called partitioned schemes, splitting the stiff hyperbolic part into an explicit non-stiff term and an implicit stiff term. However, this strategy either leads to numerical schemes suffering from a parabolic stability condition $\Delta t = \mathcal{O}(\Delta x^2)$, or to the use of implicit methods in the time integration of space gradients, which can be computationally costly because large matrices have to be inverted~\cite{Bouchut2018}. In any case, it is delicate to ensure numerical stability of implicit-explicit (IMEX) schemes for such systems~\cite{Jin1998}.

In a previous paper~\cite{wissocq2023Kinetic}, a new strategy was proposed to approximate one-dimensional convection-diffusion systems by fully explicit kinetic methods with similar stability constraints as in the hyperbolic case. This was made possible thanks to two novel ideas: (1) instead of targeting a desired system in the strict limit of a vanishing relaxation parameter, the diffusive flux is captured at first-order in a smallness parameter referred to as the Knudsen number, (2) in order to introduce the necessary free parameters for the control of diffusion, the standard BGK relaxation is replaced by a collision matrix, thereby introducing a coupling between the conserved variables within the relaxation process. The first idea is inspired by the Chapman-Enskog expansion in kinetic theory~\cite{Chapman1953, Golse2021}: the Navier-Stokes equations are not a limit of the Boltzmann equation for a vanishing Knudsen number, but a correction of this limit for low, nonzero, values of the Knudsen number. The price to pay is that we only get an approximation of the initial advection-diffusion like model, but we show in~\cite{wissocq2023Kinetic} that this approximation can be of very high quality. This strategy has been numerically evaluated on a wide variety of one-dimensional problems, including the Navier-Stokes equations which is our main application target.

The aim of this work is to extend this method to more than one dimension. As shown thereafter, this extension cannot be made by dimensional splitting because the diffusion tensor is not, in general, block diagonal. With the Navier-Stokes equations, this property is exhibited by the existence of shear viscosity: the fluid motion in one direction generates a diffusive flux in the transverse direction. Then, although the main idea of a collision matrix remains the same as in the one-dimensional case, we need to develop a slightly different technique to introduce a specific coupling between the Cartesian directions, while systematically ensuring the conservation of macroscopic variables.

The format of the paper is as follow. In Sec.~\ref{sec:problem_statement}, we first introduce the physical models we are interested in, describe the diffusive tensor and its required properties, in particular with respect to the entropy. Then, we recall the kinetic framework under consideration and in Sec.~\ref{sec:kinetic_matrix}, we introduce a new collision model based on a relaxation matrix $\Omega$ that needs to be compatible with conservation. Satisfying this fundamental constraint does not look straightforward, and this is why, inspired by what is done in the lattice Boltzmann method (LBM) community and in particular the multiple-relaxation-time (MRT) one, we rewrite the system in terms of moments~\cite{DHumieres1994, Lallemand2000, DHumieres2002}. Using this definition, the evaluation of the new collision kernel is much easier and simplification can be achieved by introducing a regularization of high-order moments, inspired by what is sometimes done in the LBM~\cite{Skordos1993, Ladd2001, Latt2006, Malaspinas2015, Coreixas2017}. 
The fully explicit numerical scheme is developed in full details in Sec.~\ref{sec:discretizations_IMEX} and numerous validations, including for the Navier-Stokes equations, are given in Sec.~\ref{sec:numerical_validations}. They show that the performances of the method are excellent with kinetic speeds obeying nothing but a subcharacteristic condition along with a hyperbolic stability constraint on the time step.

\section{Problem statement}
\label{sec:problem_statement}

\subsection{Target convection-diffusion problem}

We want to approximate the following convection-diffusion equation on $\bbu$: $\mathcal{D} \times \mathbb{R}_+ \rightarrow \mathbb{R}^p$:
\begin{align}
	\dpar{\bbu}{t} + \sum_{i=1}^d \dpar{\bbf_i(\bbu)}{x_i} = \sum_{i=1}^d \dpar{}{x_i} \left( \sum_{j=1}^d \mathbf{D}_{ij} \dpar{\bbu}{x_j} \right),
	\label{eq:target_convection_diffusion}
\end{align}
where $\mathcal{D}\subset \mathbb{R}^d$ is open, $d$ is the number of space dimensions, $p$ is the number of conserved variables in $\bbu$, $\bbf_i(\bbu)$ is the flux in the direction $x_i$ and $\mathbf{D}_{ij}$ are diffusion matrices.  The flux and diffusion matrices are assumed to be defined on $\mathcal{D}$. Note that all vectors quantities are written in bold text. It is noteworthy that \eqref{eq:target_convection_diffusion} can be written as
\begin{align}
	\dpar{\bbu}{t} + \boldsymbol{\nabla}_{\boldsymbol{x}} \cdot
	\begin{pmatrix}
		\bbf_1 (\bbu) \\ \vdots \\ \bbf_d(\bbu)
	\end{pmatrix} =
	\boldsymbol{\nabla}_{\boldsymbol{x}} \cdot \left( \mathbf{D} \boldsymbol{\nabla}_{\boldsymbol{x}} \bbu \right),
\end{align}
where $\boldsymbol{\nabla}_x$ is the vector of space derivatives $\partial/\partial_{x_i}$ and $\mathbf{D}$ is the following block matrix
\begin{align}
	\mathbf{D} = 
	\begin{pmatrix}
		\mathbf{D}_{11} & \dots & \mathbf{D}_{1d} \\
		\vdots & \ddots & \vdots \\
		\mathbf{D}_{d1} & \dots & \mathbf{D}_{dd}
	\end{pmatrix} \in \mathcal{M}_{dp}(\mathbb{R}).
\end{align}

The systems we are interested in are equipped with an entropy $\eta$, \textit{i.e.} a strictly convex function of the variables $\bbu\in \Omega$. 
This means that there exists an entropy flux
$$\bbg:=( \bbg_1(\bbu), \ldots, \bbg_d(\bbu))^T,$$ such that for any $i=1, \ldots, d$,  the functions $\bbg:\mathcal{D}\mapsto \R$ satisfy
$$\nabla_\bbu \eta^T \nabla_\bbu \bbf_i(\bbu)=\nabla_\bbu \bbg_i(\bbu).$$
We also assume that the diffusive part is compatible with the entropy, \textit{i.e.} if $\bA_0$ is the Hessian of $\eta$ with respect to $\bbu$,
$$\begin{pmatrix}
W_1\\ \vdots \\ W_d\end{pmatrix}^T 	\begin{pmatrix}
		\mathbf{D}_{11}\bA_0^{-1} & \dots & \mathbf{D}_{1d}\bA_0^{-1} \\
		\vdots & \ddots & \vdots \\
		\mathbf{D}_{d1}\bA_0^{-1} & \dots & \mathbf{D}_{dd}\bA_0^{-1}
	\end{pmatrix} 
	\begin{pmatrix}
W_1\\ \vdots \\ W_d\end{pmatrix}\geq 0$$
for any $W_j$.

\begin{example}[Two-dimensional isotropic advection-diffusion equation]
\label{ex:ADE_2D}
	For the scalar ($p=1$) advection-diffusion equation in two dimensions ($d=2$), we have
	\begin{align}
		\begin{pmatrix}
			f_1(u) \\ f_2(u)
		\end{pmatrix} =
		\begin{pmatrix}
			c_1 \\ c_2
		\end{pmatrix} u,
		\qquad
		\mathbf{D} = 
		\begin{pmatrix}
			\alpha & 0 \\ 0 & \alpha
		\end{pmatrix},
	\end{align}
	where $c_1$, $c_2$ are respectively the horizontal and vertical constant advection velocities and $\alpha\geq 0$ is a constant diffusion parameter. The entropy is $\eta=u^2/2$.
\end{example}

\begin{example}[Two-dimensional compressible Navier-Stokes equations]
\label{ex:NS_2D}
	For the compressible Navier-Stokes equations in two dimensions ($d=2$), we consider $p=4$,
	\begin{align}
		\bbu = 
		\begin{pmatrix}
			\rho \\ \rho u \\ \rho v \\ E
		\end{pmatrix}, 
		\qquad
		\bbf_1(\bbu) = 
		\begin{pmatrix}
			\rho u \\ \rho u^2 + P \\ \rho u v \\ (E+P)u
		\end{pmatrix},
		\qquad
		\bbf_2(\bbu) =
		\begin{pmatrix}
			\rho v \\ \rho u v \\ \rho v^2 + P \\ (E+P)v
		\end{pmatrix}, 
	\end{align}
	where $P$ is the thermodynamic pressure, related to $\bbu$ through an adequate equation of state. The diffusion matrix reads
	\begin{align}
		\mathbf{D} =
		\begin{pmatrix}
			\mathbf{D}_{11} & \mathbf{D}_{12} \\
			\mathbf{D}_{21} & \mathbf{D}_{22}
		\end{pmatrix},
	\end{align}
	with
	\begin{align}
		& \mathbf{D}_{11} = \frac{1}{\rho}
		\begin{pmatrix}
			0 & 0 & 0 & 0 \\ 
			-(2\mu + \lambda)u & 2\mu+\lambda & 0 & 0 \\
			-\mu v & 0 & \mu & 0 \\
			-(2\mu+\lambda)u^2 - \mu v^2 - \gamma \mu (E/\rho-u^2-v^2)/\mathrm{Pr} & (2\mu+\lambda-\gamma \mu/\mathrm{Pr})u & \mu(1-\gamma/\mathrm{Pr}) v & \gamma \mu/\mathrm{Pr}
		\end{pmatrix}, \nonumber \\ 
		& \mathbf{D}_{12} = \frac{1}{\rho}
		\begin{pmatrix}
			0 & 0 & 0 & 0 \\
			-\lambda v & 0 & \lambda & 0 \\
			-\mu u & \mu & 0 & 0 \\
			-(\mu + \lambda) u v & \mu v & \lambda u 
		\end{pmatrix}, \qquad
		\mathbf{D}_{21} = \frac{1}{\rho}
		\begin{pmatrix}
			0 & 0 & 0 & 0 \\
			-\mu v & 0 & \mu & 0 \\
			-\lambda u & \lambda & 0 & 0 \\
			-(\mu+\lambda) u v & \lambda v & \mu u & 
		\end{pmatrix} \nonumber \\
		& \mathbf{D}_{22} = \frac{1}{\rho}
		\begin{pmatrix}
			0 & 0 & 0 & 0 \\
			-\mu u & \mu & 0 & 0 \\
			-(2\mu + \lambda) v & 0 & 2\mu+\lambda & 0 \\
			-(2\mu + \lambda) v^2 - \mu u^2 - \gamma \mu(E/\rho - u^2 - v^2)/\mathrm{Pr} & \mu(1-\gamma/\mathrm{Pr}) u & (2\mu + \lambda - \gamma \mu/\mathrm{Pr}) v & \gamma \mu/\mathrm{Pr}
		\end{pmatrix},
	\end{align}
	where $\mu$ is the dynamic viscosity, $\lambda$ is the second viscosity, $\gamma$ is the heat capacity ratio and $\mathrm{Pr}$ is the Prandtl number. Under the Stokes' hypothesis, it is common to set $\lambda = -2\mu/3$. This system is compatible with the physical entropy, and among all entropies described in \cite{Harten1983}, this is the only one which is also compatible with the diffusion terms, see \cite{mallet}.
\end{example}

\subsection{Kinetic model in the zero-diffusion case}

Following~\cite{Jin, Natalini, AregbaNatalini, Abgrall2023}, the hyperbolic system~\eqref{eq:target_convection_diffusion} obtained in the zero-diffusion case ($\mathbf{D}=\mathbf{0}$) can be approximated by the following multi-dimensional BGK kinetic system:
\begin{align}
	\label{eq:kinetic_with_tau}
	\dpar{\bF}{t} + \sum_{i=1}^d \Lambda_i \dpar{\bF}{x_i} = \frac{\M(\P \bF) - \bF}{\tau},
\end{align}
where $\bF: \mathcal{D} \times \mathbb{R}^+ \rightarrow \mathbb{R}^{kp}$ is commonly referred to as a distribution function for the kinetic system, $k$ is the number of kinetic waves, $\Lambda_i$ are real diagonal $kp \times kp$ matrices composed of kinetic velocities, $\tau$ is a positive relaxation time, $\mathbb{P}: \mathbb{R}^{kp} \rightarrow \mathbb{R}^p$ is a linear operator called \textit{projector} and $\M: \mathbb{R}^p \rightarrow \mathbb{R}^{kp}$ is a Lipschitz-continuous function referred to as Maxwellian. In the present work, we consider the projector operator as the following matrix:
\begin{align}
	\P = 
	\begin{pmatrix}
		\mathbf{I}_p & \dots & \mathbf{I}_p
	\end{pmatrix} \in \mathcal{M}_{p \times kp}(\mathbb{R}),
\end{align} 
where $\mathbf{I}_p$ denotes the $(p \times p)$ identity matrix. It can be shown~\cite{Natalini, AregbaNatalini} that when the Maxwellian state is built such that 
\begin{align}
\label{eq:conditions_maxwellian}
	\begin{cases}
		\P \M(\P \bF)=\P \bF, \\
		\P \Lambda_i \M (\P \bF) = \bbf_i(\P \bF), \qquad i=1 \dots d,
	\end{cases}
\end{align} 
then the hyperbolic system \eqref{eq:target_convection_diffusion} with $\mathbf{D}=\mathbf{0}$ is the formal limit of \eqref{eq:kinetic_with_tau} when $\tau \rightarrow 0$. These relations naturally lead to a condition on the number of waves to be considered in the kinetic model: $k \geq d+1$. Furthermore, when the kinetic velocities in $\Lambda_i$ are such that the Maxwellian functions are monotone in the sense of~\cite{Bouchut}, then the kinetic system is compatible with entropy inequalities.

\begin{example}[Four-wave model for the two-dimensional advection equation]
\label{ex:kinetic_advection_2D}
	Considering the notations of Example~\ref{ex:ADE_2D}, a four-wave model ($k=4$) can be built with:
	\begin{align}
		\bF = 
		\begin{pmatrix}
			f_1 \\ f_2 \\ f_3 \\ f_4
		\end{pmatrix}, \qquad
		\Lambda_1 = 
		\begin{pmatrix}
			-a & 0 & 0 & 0 \\
			0 & a & 0 & 0 \\
			0 & 0 & 0 & 0 \\
			0 & 0 & 0 & 0
		\end{pmatrix}, \qquad
		\Lambda_2 =
		\begin{pmatrix}
			0 & 0 & 0 & 0 \\
			0 & 0 & 0 & 0 \\
			0 & 0 & -a & 0 \\
			0 & 0 & 0 & a
		\end{pmatrix}, \qquad
		\M = 
		\begin{pmatrix}
			\M_1 \\ \M_2 \\ \M_3 \\ \M_4
		\end{pmatrix},
	\end{align}
	with $a>0$ and with
	\begin{align}
	\label{ex:advection_maxwellian_conditions}
		\begin{cases}
			\P \M  = \P \bF, \\
			\P \Lambda_1 \M = f_1(\P \bF), \\
			\P \Lambda_2 \M = f_2(\P \bF),
		\end{cases}
		\Rightarrow
		\begin{cases}
			\sum_{j=1}^k \M_j = \sum_{j=1}^k f_j, \\
			a(-\M_1 + \M_2) = f_1 (\P \bF), \\
			a(-\M_3 + \M_4) = f_2 (\P \bF).
		\end{cases}
	\end{align}
	Note that the problem \eqref{ex:advection_maxwellian_conditions} is underdetermined to build a Maxwellian. To close the system, the additional condition $\M_1 + \M_2 = \P \bF/2$ (or equivalently, $\M_3+\M_4 = \P \bF/2$) is commonly considered, leading to an isotropic Maxwellian~\cite{Natalini, Abgrall2023}. In this case, the monotonicity condition of the Maxwellian leads to the subcharacteristic condition: $a > 2\max(f_1'(\P \bF), f_2'(\P \bF))$. Finally note that a kinetic model with four waves is two dimensions is commonly referred to as \textit{D2Q4 lattice} in the lattice Boltzmann community~\cite{WolfGladrow2000}.
\end{example}

\begin{example}[Four-wave model for the two-dimensional Euler equations]
\label{ex:kinetic_Euler_2D}
	With the notations of Example~\ref{ex:NS_2D}, we define the following kinetic model with four waves ($k=4$):
	\begin{align}
		\bF = 
		\begin{pmatrix}
			\rho_1 \\ (\rho u)_1 \\ (\rho v)_1 \\ E_1 \\
			\vdots \\
			\rho_4 \\ (\rho u)_4 \\ (\rho v)_4 \\ E_4 \\
		\end{pmatrix}, \qquad
		\Lambda_1 =
		\begin{pmatrix}
			-a & 0 & 0 & 0 \\
			0 & a & 0 & 0 \\
			0 & 0 & 0 & 0 \\
			0 & 0 & 0 & 0
		\end{pmatrix} \otimes \mathbf{I}_4, \qquad
		\Lambda_2 =
		\begin{pmatrix}
			0 & 0 & 0 & 0 \\
			0 & 0 & 0 & 0 \\
			0 & 0 & -a & 0 \\
			0 & 0 & 0 & a
		\end{pmatrix} \otimes \mathbf{I}_4, \qquad
		\M = 
		\begin{pmatrix}
			\M_1^\rho \\ \M_1^{\rho u} \\ \M_1^{\rho v} \\ \M_1^E \\
			\vdots \\
			\M_4^\rho \\ \M_4^{\rho u} \\ \M_4^{\rho v} \\ \M_4^E
		\end{pmatrix},
		\end{align}
		where $a>0$ and where $\otimes$ denotes the Kronecker product between two matrices\footnote{If $\mathbf{A}=(a_{ij})\in \mathcal{M}_k(\R)$, and $\mathbf{B}\in \mathcal{M}_r(\R)$,
$\mathbf{A} \otimes \mathbf{B} =\begin{pmatrix}
a_{11}\mathbf{B} & \ldots & a_{1n} \mathbf{B}\\
\vdots & \ddots &\vdots\\
a_{k1} \mathbf{B} &\ldots &a_{kk} \mathbf{B}\end{pmatrix}\in \mathcal{M}_{kr}(\R)$.}. As in Example~\ref{ex:kinetic_advection_2D}, the system \eqref{eq:conditions_maxwellian} is underdetermined with a four-wave model and has to be supplemented by additional constraints. Considering $\M_1^\rho + \M_2^\rho = \rho/2$ and similarly for the other variables of $\bbu$, the monotonicity of the Maxwellian leads to the following subcharacteristic condition: $a > 2 \max[\rho(\bbf_1'(\P \bF)), \rho(\bbf_2'(\P \bF))]$, where $\rho(\mathbf{M})$ denotes the spectral radius of a matrix $\mathbf{M}$.
\end{example}

In the present work, we want to follow the strategy adopted in~\cite{wissocq2023Kinetic} to approach, with an arbitrary controllable error, the convection-diffusion problem \eqref{eq:target_convection_diffusion} with a kinetic method. Instead of studying the limit of the kinetic system when $\tau \rightarrow 0$, we want to perform an asymptotic expansion for ``small'' values of the relaxation time, and target the diffusive flux of \eqref{eq:target_convection_diffusion} at first-order in this expansion. Since ``small' by itself is meaningless, we need to introduce dimensionless numbers in order to quantity how the relaxation time can be small compared to other characteristic times of the kinetic problem. For this reason, and following the introduction of Knudsen number in kinetic theory of gases~\cite{golse:hal-00859451, Golse2021}, we consider a characteristic length scale of the problem under consideration $\ell$. We assume an isotropy of the kinetic velocities, which implies that the matrices $\Lambda_i$ have the same $L^2$ norm that can be written as $||\Lambda|| = ||\Lambda_i||$. This allows us to define dimensionless time, space and velocity matrix as
\begin{align}
	t^* = \frac{||\Lambda|| t}{\ell}, \qquad x^* = \frac{x}{\ell}, \qquad \Lambda_i^* = \frac{\Lambda_i}{||\Lambda||}.
\end{align}
The dimensionless BGK system reads
\begin{align}
	\dpar{\bF}{t^*} + \sum_{i=1}^d \Lambda_i^* \dpar{\bF}{x_i^*} = \frac{\M(\P \bF)- \bF}{\varepsilon},
\end{align}
where $\varepsilon$ is the Knudsen number defined as
\begin{align}
	\varepsilon = \frac{||\Lambda|| \tau}{\ell}.
\end{align}
Assuming that $\varepsilon \ll 1$ allows us to perform an asymptotic expansion in $\varepsilon$ (namely a Chapman-Enskog expansion~\cite{Chapman1953}) and neglect high-order terms. In the rest of the article, considering the notations adopted in~\cite{wissocq2023Kinetic}, we will adopt the following, maybe less familiar, form of \eqref{eq:kinetic_with_tau}:
\begin{align}
	\dpar{\bF}{t} + \sum_{i=1}^d \Lambda_i \dpar{\bF}{x_i} = \frac{||\Lambda||}{\ell} \frac{\M(\P \bF) - \bF}{\varepsilon}.
	\label{eq:kinetic_with_epsilon}
\end{align}
Similarly to what is done in~\cite{Natalini, wissocq2023Kinetic}, an asymptotic expansion in $\varepsilon$ yields, up to the second-order (see App.~\ref{app:CE_BGK} for the details of the Chapman-Enskog expansion):
\begin{align}
	\label{eq:CE_BGK}
	\dpar{\bbu^\varepsilon}{t} + \sum_{i=1}^d \dpar{\bbf_i(\bbu^\varepsilon)}{x_i} = \varepsilon \frac{\ell}{||\Lambda||} \sum_{i=1}^d \dpar{}{x_i} \left\{ \sum_{j=1}^d \left[ \P \Lambda_i \Lambda_j \M'(\bbu^\varepsilon) - \bbf_i'(\bbu^\varepsilon) \bbf_j'(\bbu^\varepsilon) \right] \dpar{\bbu^\varepsilon}{x_j} \right\} + \mathcal{O}(\varepsilon^2),
\end{align}
where $\bbu^\varepsilon := \P \bF$. It can be shown that, when $\M$ is a monotone Maxwellian function in the sense of~\cite{Bouchut}, then the approximated right-hand-side term of \eqref{eq:CE_BGK} dissipates all the Lax entropies of the kinetic BGK system. Eq.~\eqref{eq:CE_BGK} can be seen as an approximation of the convection-diffusion problem \eqref{eq:target_convection_diffusion} with a very specific diffusion matrix
\begin{align}
	\mathbf{D}_{ij} = \varepsilon \frac{\ell}{||\Lambda||} \left[ \P \Lambda_i \Lambda_j \M'(\bbu^\varepsilon) - \bbf_i'(\bbu^\varepsilon) \bbf_j'(\bbu^\varepsilon) \right].
\end{align}
As discussed in~\cite{wissocq2023Kinetic}, this formalism does not allow us to approximate a desired convection-diffusion problem for a given diffusion matrix $\mathbf{D}$. To this extent, the purpose of the next section is to introduce a new kinetic model involving a collision matrix for multi-dimensional convection-diffusion systems based on the idea developed in \cite{wissocq2023Kinetic}.

\section{Kinetic model with a collision matrix}
\label{sec:kinetic_matrix}

As proposed in \cite{wissocq2023Kinetic} in the one-dimensional case, we modify the BGK framework of \eqref{eq:kinetic_with_epsilon} by introducing a square collision matrix $\Omega \in \mathcal{M}_{kp}(\mathbb{R})$, as
\begin{align}
	\dpar{\bF}{t} + \sum_{i=1}^d \Lambda_i \dpar{\bF}{x_i} = \frac{||\Lambda||}{\ell} \frac{\Omega}{\varepsilon}(\M(\P \bF) - \bF).
	\label{eq:kinetic_with_Omega}
\end{align}

Note that $\Omega$ is a matrix to be defined. As a consequence, the Knudsen number $\varepsilon$ has to be defined as well. We first discuss some properties that have to be satisfied by the collision matrix allowing us to make some assumptions on its form.

\subsection{Conservation properties of the collision matrix}

A first property that has to be satisfied by $\Omega$ is that $\bbu^\varepsilon = \P \bF$ remains a collision invariant of the kinetic system, ensuring its conservation. This reads
\begin{align}
	\P \Omega (\M (\P \bF) - \bF) = \mathbf{0}.
\end{align} 
Let us build a general framework allowing us to ensure this property whatever the wave model under consideration. To do so, we draw inspiration from the multiple relaxation time (MRT) models in the lattice Boltzmann community~\cite{DHumieres1994, Lallemand2000, DHumieres2002} and switch to a formulation of the stiff relaxation in the moment space. 

It is necessary to introduce a linear transformation between the vector $\bF$ and a vector containing $\bbu^\varepsilon=\P\bF$, $\bbv_i^\varepsilon:=\P\Lambda_i\bF$, $i=1, \ldots, d$. Since $\bF\in \R^{kp}$ and $\bbu^\varepsilon, \bbv_i^\varepsilon \in \R^p$, then $\big (\bbu^\varepsilon, \bbv_1^\varepsilon, \ldots , \bbv_d^\varepsilon \big )\in \R^{(d+1)p}$. Recall that we necessarily have $k \geq d+1$ in order to construct a wave model. Thus, we see that we need to complement the vector $\big (\bbu^\varepsilon, \bbv_1^\varepsilon, \ldots , \bbv_d^\varepsilon \big )$ by a vector $\mathbf{h}$ of $\R^{p(k-d-1)}$. We can make the mapping $\bF\mapsto \big ( \bbu^\varepsilon, \bbv_1^\varepsilon, \ldots , \bbv_d^\varepsilon , \mathbf{h}\big )$ linear~\footnote{\label{footproof}In term of linear algebra, let us denote by $E$ the vector space generated by the 
 $\big (\bbu^\varepsilon, \bbv_1^\varepsilon, \ldots , \bbv_d^\varepsilon, \mathbf{0}_{p(k-d-1)} \big )$. It is of dimension $(d+1)p$ and  admits a supplement $H$ of dimension $p(k-d-1)$, $\R^{pk}=E\oplus H$. The mapping $\Psi: \bF\mapsto (\bbu^\varepsilon, \bbv_1^\varepsilon, \ldots , \bbv_d^\varepsilon):=\Psi(\bF)$ 
 is linear and maps $\R^{kp}$ to $E$. Then we can consider a projector $p$ of $\R^{kp}$ onto $H$, and the matrix $\momentsmatrix$ can be seen as $\momentsmatrix\bF=\big (\Psi(\bF), p(\bF) \big )$.}.  
 For reasons that will become more clear later in the text, and following the LBM community terminology, the additional vector $\mathbf{h}$ will be referred to as "high-order moments". We therefore define a moments matrix as follows.

\begin{definition}[Moments matrix]
	Considering a kinetic system in the form of \eqref{eq:kinetic_with_Omega}, we define a moments matrix $\momentsmatrix$ as
\begin{align}
\label{eq:moments_matrix}
	\momentsmatrix = 
	\begin{pmatrix}
		\P \\
		\P \Lambda_1 \\
		\vdots \\
		\P \Lambda_d \\
		\mathbb{H}
	\end{pmatrix},
\end{align}
where $\mathbb{H} \in \mathcal{M}_{hp \times kp}(\mathbb{R})$ is an arbitrary matrix designed to compute the high-order moments $\mathbf{h}$ from $\bF$ and $h=k-d-1 \geq 0$ is the dimension of high-order moments. This matrix is built such that $\mathrm{det}(\momentsmatrix) \neq 0$, \textit{i.e.} it is invertible by definition.
\end{definition}

\begin{definition}[Moments]
	With a given moments matrix, we define moments of the distributions $\bF$ and of the Maxwellian $\M (\P \bF)$ as
\begin{align}
	\mathbf{m}^{\bF} = \momentsmatrix \bF, \qquad \mathbf{m}^{\M} = \momentsmatrix \M (\P \bF).
\end{align}
\end{definition}

As such, the matrix $\momentsmatrix$ can be used to create a bijection between distribution functions $\bF$ and the following moments:
\begin{itemize}
	\item by application of $\P$, we compute the \emph{zeroth-order moments} of $\bF$,
	\item by application of $\P \Lambda_i$, we compute the \emph{first-order moments} of $\bF$,
	\item by application of $\mathbb{H}$, we compute the \emph{high-order moments} of $\bF$.
\end{itemize}
Note that with our convention, the projector $\P$ can be extracted from $\momentsmatrix$ as
\begin{align}
	\label{eq:PfromM}
	\P = 
	\begin{pmatrix}
		\mathbf{I}_p & \mathbf{0}_{p\times(k-1)p}
	\end{pmatrix}
	 \momentsmatrix,
\end{align}
where $\mathbf{0}_{m \times n}$ is the null matrix of size $m \times n$. Similarly, the first-order projector can be extracted as
\begin{align}
\label{eq:P1fromM}
	\begin{pmatrix}
		\P \Lambda_1 \\ \vdots \\ \P \Lambda_d
	\end{pmatrix} = \begin{pmatrix}
		\mathbf{0}_{dp \times p} & \mathbf{I}_{dp} & \mathbf{0}_{dp \times hp}
	\end{pmatrix}
	\momentsmatrix.
\end{align}  
The moments matrix can be used to build a Maxwellian satisfying the conditions \eqref{eq:conditions_maxwellian} in a systematic way. When the Maxwellian is known as a linear function of the variables $(\bbu^\varepsilon, \bbf_1(\bbu^\varepsilon), \dots \bbf_d(\bbu^\varepsilon))$, the following proposition can be demonstrated.

\begin{proposition}
\label{prop:HMzero}
	If $\M$ is a linear function of the variables $(\bbu^\varepsilon, \bbf_1(\bbu^\varepsilon), \dots, \bbf_d(\bbu^\varepsilon))$, then there exists a moments matrix $\momentsmatrix$ such that $\mathrm{det}(\momentsmatrix) \neq 0$ with $\mathbb{H}\M(\P \bF) = \mathbf{0}$.
\end{proposition}
\begin{proof}
	Suppose that we know a moments matrix $\momentsmatrix_*$ of the kinetic system with $\mathrm{det}(\momentsmatrix_*) \neq 0$, see the footnote \ref{footproof} for a proof of the existence. 
	 If $\M$ is a linear function of the variables $(\bbu^\varepsilon, \bbf_1(\bbu^\varepsilon), \dots, \bbf_d(\bbu^\varepsilon))$, then $\mathbf{m}_*^\M := \momentsmatrix_* \M$ is also a linear function of these variables. Without loss of generality, this reads:
	\begin{align}
		\momentsmatrix_* =
		\begin{pmatrix}
			\P \\ \P \Lambda_1 \\ \vdots \\ \P \Lambda_d \\ \mathbf{L}^*_1 \\ \vdots \\ \mathbf{L}^*_{h}
		\end{pmatrix}, \qquad
		\mathbf{m}_*^\M =
		\begin{pmatrix}
			\bbu^\varepsilon \\ 
			\bbf_i(\bbu^\varepsilon) \\
			\vdots \\
			\bbf_d(\bbu^\varepsilon) \\
			\mu_{10}  \bbu^\varepsilon + \sum_{j=1}^d \mu_{1j} \bbf_j (\bbu^\varepsilon) 
			\\ \vdots \\
			\mu_{h0} \bbu^\varepsilon + \sum_{j=1}^d \mu_{hj} \bbf_j(\bbu^\varepsilon)
		\end{pmatrix},
	\end{align}
where $\mathbf{L}^*_i \in \mathcal{M}_{p\times kp}(\mathbb{R}) $ and $\mu_{ij}$ are real coefficients. Then define, for all $i \in \{1, \dots h \}$,
\begin{align}
	\mathbf{L}_i = \mathbf{L}^*_i - \mu_{i0} \P - \sum_{j=1}^d \mu_{ij} \P \Lambda_j.  
\end{align} 
We can easily check that $\mathbf{L}_i \M (\bbu^\varepsilon) = \mathbf{0}$. Now define a new moments matrix $\momentsmatrix$ as \eqref{eq:moments_matrix} where $\mathbb{H}$ is composed of the lines $\mathbf{L}_i$. We have $\mathbb{H} \M(\bbu^\varepsilon) = \mathbf{0}$ and $\mathrm{det}(\momentsmatrix) = \mathrm{det}(\momentsmatrix_*) \neq 0$, which ends the proof.
\end{proof}

We now illustrate the linear mapping with some examples.

\begin{example}[Four-wave model in two dimensions]
\label{ex:moments_matrix_D2Q4}
	We consider the general four-wave model for a system of $p$ conservation equations in two dimensions with
	\begin{align}
		\Lambda_1 = 
		\begin{pmatrix}
			-a & 0 & 0 & 0 \\
			0 & a & 0 & 0 \\
			0 & 0 & 0 & 0 \\
			0 & 0 & 0 & 0
		\end{pmatrix} \otimes \mathbf{I}_p , \qquad 
		\Lambda_2 = 
		\begin{pmatrix}
			0 & 0 & 0 & 0 \\
			0 & 0 & 0 & 0 \\
			0 & 0 & -a & 0 \\
			0 & 0 & 0 & a
		\end{pmatrix} \otimes \mathbf{I}_p,
		.
	\end{align}
	The dimension of high-order moments is $h=k-d-1=1$. The consistency equations \eqref{eq:conditions_maxwellian} supplemented by the condition $2(\M_1 + \M_2) = \bbu^\varepsilon$ read
	\begin{align}
		\mathbf{m}_*^\M :=  \momentsmatrix_* \M(\bbu^\varepsilon) = 
		\begin{pmatrix}
			\bbu^\varepsilon \\
			\bbf_1(\bbu^\varepsilon) \\
			\bbf_2(\bbu^\varepsilon) \\
			\bbu^\varepsilon
		\end{pmatrix}, \qquad
		\momentsmatrix_* = 
		\begin{pmatrix}
			1 & 1 & 1 & 1 \\
			-a & a & 0 & 0 \\
			0 & 0 & -a & a \\
			2 & 2 & 0 & 0 
		\end{pmatrix}\otimes \mathbf{I}_p,
	\end{align}
	and $\mathrm{det}(\momentsmatrix_*) = 8a^2 \neq 0$. These conditions lead to a unique Maxwellian,
	\begin{align}
		\M(\bbu^\varepsilon) = \momentsmatrix_*^{-1} \mathbf{m}_*^\M = \frac{1}{2}
		\begin{pmatrix}
			\bbu^\varepsilon/2 - \bbf_1(u^\varepsilon)/a \\
			\bbu^\varepsilon/2 + \bbf_1(u^\varepsilon)/a \\
			\bbu^\varepsilon/2 - \bbf_2(u^\varepsilon)/a \\
			\bbu^\varepsilon/2 + \bbf_2(u^\varepsilon)/a
		\end{pmatrix},
	\end{align}
	which is the isotropic Maxwellian state adopted in previous work with the four-wave model~\cite{Natalini, AregbaNatalini, Abgrall2023}. Following proposition~\ref{prop:HMzero}, it is possible to define another moments matrix $\momentsmatrix$ as
	\begin{align}
		\momentsmatrix = 
		\begin{pmatrix}
			1 & 1 & 1 & 1 \\
			-a & -a & 0 & 0 \\
			0 & 0 & -a & a \\
			1 & 1 & -1 & -1
		\end{pmatrix} \otimes \mathbf{I}_p,
	\end{align}
	with $\mathbb{H} = [1, 1, -1, -1] \otimes \mathbf{I}_p$ and check that $\mathbb{H} \M (\bbu^\varepsilon) = \mathbf{0}$ and $\mathrm{det}(\momentsmatrix) = \mathrm{det}(\momentsmatrix_*) \neq 0$. 
\end{example}

\begin{remark}
	With the definition of $\momentsmatrix$ in Example~\ref{ex:moments_matrix_D2Q4}, we see that $\mathbb{H} = \P (\Lambda_1^2-\Lambda_2^2)/a^2$, which justifies the denomination ``high-order moment'' for $\mathbf{h}=\mathbb{H} \bF$. In this example, it is a combination of second-order moments.
\end{remark}

\begin{example}[Eight-velocity model in two dimensions]
	\label{ex:moments_matrix_D2Q8}
	We consider an eight-velocity kinetic model where
	\begin{align}
		& \Lambda_1 = \mathrm{diag}((a_i \cos(\theta_i))_{i \in \{1,.., 8\}}) \otimes \mathbf{I}_p, \qquad \Lambda_2 = \mathrm{diag}((a_i \sin(\theta_i))_{i \in \{1,.., 8\}}) \otimes \mathbf{I}_p , \\
		& a_1 = a_3 = a_5 = a_7 = a, \qquad a_2 = a_4 = a_6 = a_8 = a\sqrt{2}, \qquad \theta_i = (i-1) \pi/4.
	\end{align}
	This velocity model, displayed in Fig.~\ref{fig:D2Q8}, is sometimes referred to as D2Q8 lattice in the LBM community~\cite{Junk2005}. Following \cite{AregbaNatalini}, we can build a Maxwellian as
	\begin{align}
		\M_i = \frac{\bbu^\varepsilon}{8} + \frac{a_i}{6a^2}\left[ \cos(\theta_i) \bbf_1(\bbu^\varepsilon) + \sin(\theta_i) \bbf_2(\bbu^\varepsilon) \right],
	\end{align}
	where we note $\M_i$ the Maxwellian related to the wave labelled $i$. This choice leads to
	\begin{align}
		& \M_1 = \frac{\bbu^\varepsilon}{8} + \frac{\bbf_1(\bbu^\varepsilon)}{6a}, & \M_2 = \frac{\bbu^\varepsilon}{8} + \frac{\bbf_1(\bbu^\varepsilon) + \bbf_2(\bbu^\varepsilon)}{6a},\ & \M_3 = \frac{\bbu^\varepsilon}{8} + \frac{\bbf_2(\bbu^\varepsilon)}{6a}, & \M_4 = \frac{\bbu^\varepsilon}{8} + \frac{-\bbf_1(\bbu^\varepsilon) + \bbf_2(\bbu^\varepsilon)}{6a}, \\
		& \M_5 = \frac{\bbu^\varepsilon}{8} - \frac{\bbf_1(\bbu^\varepsilon)}{6a}, & \M_6 = \frac{\bbu^\varepsilon}{8} - \frac{\bbf_1(\bbu^\varepsilon) + \bbf_2(\bbu^\varepsilon)}{6a},\ & \M_7 = \frac{\bbu^\varepsilon}{8} - \frac{\bbf_2(\bbu^\varepsilon)}{6a}, & \M_8 = \frac{\bbu^\varepsilon}{8} + \frac{\bbf_1(\bbu^\varepsilon) - \bbf_2(\bbu^\varepsilon)}{6a}.
	\end{align}
	The dimension of high-order moments is $h=k-d-1=5$. We can define the following moments matrix $\momentsmatrix$ leading to the following Maxwellian moments $\mathbf{m}^\M$:
	\begin{align}
		\momentsmatrix = 
		\begin{pmatrix}
			1 & 1 & 1 & 1 & 1 & 1 & 1 & 1 \\
			a & a & 0 & -a & -a & -a & 0 & a \\
			0 & a & a & a & 0 & -a & -a & -a \\
			a^2 & a^2 & -3a^2 & a^2 & a^2 & a^2 & -3a^2 & a^2 \\
			-3a^2 & a^2 & a^2 & a^2 & -3a^2 & a^2 & a^2 & a^2 \\
			0 & a^2 & 0 & -a^2 & 0 & a^2 & 0 & -a^2 \\
			0 & a^3 & -2a^3 & a^3 & 0 & -a^3 & 2a^3 & -a^3 \\
			-2a^3 & a^3 & 0 & -a^3 & 2a^3 & -a^3 & 0 & a^3
		\end{pmatrix} \otimes \mathbf{I}_p , \qquad
		\mathbf{m}^\M = 
		\begin{pmatrix}
			\bbu^\varepsilon \\
			\bbf_1(\bbu^\varepsilon) \\
			\bbf_2(\bbu^\varepsilon) \\
			\mathbf{0} \\ \mathbf{0} \\ \mathbf{0} \\ \mathbf{0} \\ \mathbf{0} 
		\end{pmatrix}.
	\end{align}
	We see that this choice of $\momentsmatrix$ leads to $\mathbb{H} \M(\P \bF) = \mathbf{0}$. Moreover, note that the matrix $\mathbb{H}$ can be written as 
	\begin{align}
		\mathbb{H} = 
		\begin{pmatrix}
			\P ( 4\Lambda_1^2 - 3a^2 \mathbf{I}_{kp}) \\
			\P (4\Lambda_2^2 - 3a^2 \mathbf{I}_{kp}) \\
			\P \Lambda_1 \Lambda_2 \\
			\P (3\Lambda_1^2 \Lambda_2 - 2a^2 \Lambda_2) \\
			\P (3\Lambda_2^2 \Lambda_1 - 2a^2 \Lambda_1)
		\end{pmatrix},
	\end{align}
	\textit{i.e.} as high-order powers of the kinetic matrices $\Lambda_1$ and $\Lambda_2$, thus justifying the denomination ``high-order moments''.
	\end{example}

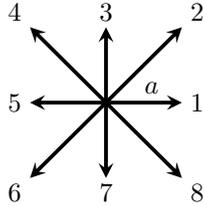
\begin{figure}[h!]
\centering
    \begin{tikzpicture}
        \begin{scope}[xshift=1cm,yshift=-1.2cm]
            \draw [line width=0.5mm,,>=stealth, ->] (0,0) -- (1,0);
            \draw [line width=0.5mm,>=stealth, ->] (0,0) -- (0,1);
            \draw [line width=0.5mm,>=stealth, ->] (0,0) -- (-1,0);
            \draw [line width=0.5mm,>=stealth, ->] (0,0) -- (0,-1);
            \draw [line width=0.5mm,,>=stealth, ->] (0,0) -- (1,1);
            \draw [line width=0.5mm,,>=stealth, ->] (0,0) -- (-1,1);
            \draw [line width=0.5mm,,>=stealth, ->] (0,0) -- (-1,-1);
            \draw [line width=0.5mm,,>=stealth, ->] (0,0) -- (1,-1);
	        \node at (0.6,0.2){$a$};
	        \node at (1.2,0){$1$};
	        \node at (1.2,1.2){$2$};
	        \node at (0, 1.2){$3$};
	        \node at (-1.2, 1.2){$4$};
	        \node at (-1.2, 0){$5$};
         	\node at (-1.2, -1.2){$6$};
	        \node at (0, -1.2){$7$};
	        \node at (1.2, -1.2){$8$};
        \end{scope}
    \end{tikzpicture}
    \caption{Eight-velocity model (D2Q8).}
    \label{fig:D2Q8}
\end{figure}

In the following, we will always consider a moment matrix $\momentsmatrix$ built such that $\mathbb{H}\M (\P \bF) = \mathbf{0}$, for the sake of simplicity. With these notations, we note $\mathbf{C}$ the collision matrix written in the moments basis, \textit{i.e.}
\begin{align}
\label{eq:Omega_C}
	\Omega = \momentsmatrix^{-1} \mathbf{C} \momentsmatrix.
\end{align}
The definition of the collision matrix $\Omega$ comes down to the definition of $\mathbf{C}$. In the following, we look for a collision matrix in the moments basis under the block matrix form
\begin{align}
	\label{eq:def_C_alpha}
	\mathbf{C} = 
	\begin{pmatrix}
		\alpha \mathbf{I}_{p} & \mathbf{0}_{p \times dp} & \mathbf{0}_{p\times hp} \\
		\mathbf{0}_{dp \times p} & \tilde{\mathbf{C}} & \mathbf{0}_{dp \times hp} \\
		\mathbf{0}_{hp \times p} & \mathbf{0}_{hp \times dp} & \alpha \mathbf{I}_{hp}
	\end{pmatrix},
\end{align}
where $\alpha >0$ is a free coefficient and $\tilde{\mathbf{C}} \in \mathcal{M}_{dp}(\mathbb{R})$ is a square matrix to be defined. This form of $\mathbf{C}$ assumes a relaxation of zeroth- and high-order moments with the scalar parameter $\alpha$, while first-order moments are relaxed with the collision matrix $\tilde{\mathbf{C}}$.  With this assumption and using \eqref{eq:PfromM}, we can check that
\begin{align}
	\P \Omega = \P \momentsmatrix^{-1} \mathbf{C} \momentsmatrix = \alpha \P,
\end{align}
so that $\P \Omega (\M(\P \bF) - \bF) = \mathbf{0}$ and conservation can be systematically ensured. Eventually left-multiplying \eqref{eq:kinetic_with_Omega} by $\P$ yields the following conservation equation for $\bbu^\varepsilon = \P \bF$:
\begin{align}
\label{eq:conservation_uepsilon}
	\dpar{\bbu^\varepsilon}{t} + \sum_{i=1}^d \dpar{(\P \Lambda_i \bF)}{x_i} = \mathbf{0}.
\end{align}
The fluxes can then be evaluated with a Chapman-Enskog expansion as in the next section.

\subsection{Chapman-Enskog expansion}

First note that Eq.~\eqref{eq:conservation_uepsilon} can also be written
\begin{align}
	\dpar{\bbu^\varepsilon}{t} + \boldsymbol{\nabla_x} \cdot \left[ \begin{pmatrix}
		\P \Lambda_1 \\ \vdots \\ \P \Lambda_d
	\end{pmatrix} \bF \right] = \mathbf{0}.
\end{align}
Similarly to the Chapman-Enskog expansion performed in App.~\ref{app:CE_BGK}, we have from Eq.~\eqref{eq:kinetic_with_Omega},
\begin{align}	
	\label{eq:Chapman_F}
	\bF = \M(\P \bF) - \varepsilon \frac{\ell}{||\Lambda||} \Omega^{-1} \left( \dpar{\bF}{t} + \sum_{j=1}^d \Lambda_j \dpar{\bF}{x_j} \right).
\end{align}
Then, noticing that $\bF = \M (\P \bF) + \mathcal{O}(\varepsilon)$ and left-multiplying by $\P \Lambda_i$ yields
\begin{align}
	\begin{pmatrix}
		\P \Lambda_1 \\ \vdots \\ \P \Lambda_d	
	\end{pmatrix}
	\bF = 
	\begin{pmatrix}
		\mathbf{f}_1(\bbu^\varepsilon) \\ \vdots \\ \mathbf{f}_d(\bbu^\varepsilon)
	\end{pmatrix}
	- \varepsilon \frac{\ell}{||\Lambda||}
	\begin{pmatrix}
		\P \Lambda_1 \\ \vdots \\ \P \Lambda_d
	\end{pmatrix}
	\Omega^{-1} \left( \dpar{\M(\P \bF)}{t} + \sum_{j=1}^d \Lambda_j \dpar{\M(\P \bF)}{x_j} \right) + \mathcal{O}(\varepsilon^2),
\end{align}
where we used the fact that $\P \Lambda_i \M(\P \bF) = \bbf_i(\bbu^\varepsilon)$. Furthermore, using \eqref{eq:P1fromM}, we have
\begin{align}
	\label{eq:Omega_C_permutation}
	\begin{pmatrix}
		\P \Lambda_1 \\ \vdots \\ \P \Lambda_d 
	\end{pmatrix} 
	\Omega^{-1} = 
	\begin{pmatrix}
		\mathbf{0}_{dp \times p} & \mathbf{I}_{dp} & \mathbf{0}_{dp \times hp}
	\end{pmatrix}
	\mathbf{C}^{-1} \momentsmatrix = 
	\tilde{\mathbf{C}}^{-1}
	\begin{pmatrix}
		\P \Lambda_1 \\ \vdots \\ \P \Lambda_d 
	\end{pmatrix}.
\end{align}
This leads to
\begin{align}
	\begin{pmatrix}
		\P \Lambda_1 \\ \vdots \\ \P \Lambda_d	
	\end{pmatrix}
	\bF = 
	\begin{pmatrix}
		\mathbf{f}_1(\bbu^\varepsilon) \\ \vdots \\ \mathbf{f}_d(\bbu^\varepsilon)
	\end{pmatrix}
	- \varepsilon \frac{\ell}{||\Lambda||}
	\tilde{\mathbf{C}}^{-1} \left\{ \dpar{}{t}
	\begin{pmatrix}
		\bbf_1(\bbu^\varepsilon) \\ \vdots \\ \bbf_d(\bbu^\varepsilon)
	\end{pmatrix} + \sum_{j=1}^d 
	\begin{pmatrix}
		\P \Lambda_1 \\ \vdots \\ \P \Lambda_d
	\end{pmatrix} \Lambda_j \dpar{\M(\P \bF)}{x_j} \right\} + \mathcal{O}(\varepsilon^2).
\end{align}
Using similar chain rules as in App.~\ref{app:CE_BGK} gives
\begin{align}
	& \dpar{}{t}
	\begin{pmatrix}
		\bbf_1(\bbu^\varepsilon) \\ \vdots \\ \bbf_d(\bbu^\varepsilon)
	\end{pmatrix} =  - 
	\mathbf{J}_{\bbf}
	\boldsymbol{\nabla}_{\boldsymbol{x}} \bbu^\varepsilon + \mathcal{O}(\varepsilon),  \qquad \sum_{j=1}^d 
	\begin{pmatrix}
		\P \Lambda_1 \\ \vdots \\ \P \Lambda_d
	\end{pmatrix}
	\Lambda_j \dpar{\M(\P \bF)}{x_j} =
	\mathbf{J}_{\Lambda}
	\boldsymbol{\nabla}_{\boldsymbol{x}} \bbu^\varepsilon,
\end{align}
with
\begin{align}
	\mathbf{J}_{\bbf} := 
	\begin{pmatrix}
		(\mathbf{f}_1'(\bbu^\varepsilon))^2 & \hdots & \mathbf{f}_1'(\bbu^\varepsilon) \mathbf{f}_d'(\bbu^\varepsilon) \\
		\vdots & \ddots & \vdots \\
		\mathbf{f}_d'(\bbu^\varepsilon) \mathbf{f}_1'(\bbu^\varepsilon) & \hdots & (\mathbf{f}_d'(\bbu^\varepsilon))^2
	\end{pmatrix}, \qquad \mathbf{J}_\Lambda :=
	\begin{pmatrix}
		\P \Lambda_1^2 \M'(\bbu^\varepsilon) & \hdots & \P \Lambda_1 \Lambda_d \M'(\bbu^\varepsilon) \\ 
		\vdots & \ddots & \vdots \\
		\P \Lambda_d \Lambda_1 \M'(\bbu^\varepsilon) & \hdots & \P \Lambda_d^2 \M'(\bbu^\varepsilon)
	\end{pmatrix}.
\end{align}
Finally, we obtain
\begin{align}
	\dpar{\bbu^\varepsilon}{t} + \boldsymbol{\nabla}_{\boldsymbol{x}} \cdot
	\begin{pmatrix}
		\bbf_1(\bbu^\varepsilon) \\ \vdots \\ \bbf_d(\bbu^\varepsilon)
	\end{pmatrix}
	= \varepsilon \boldsymbol{\nabla}_{\boldsymbol{x}} \cdot \left\{ \frac{\ell}{||\Lambda||} \tilde{\mathbf{C}}^{-1}
	(\mathbf{J}_{\Lambda} - \mathbf{J}_{\bbf} ) \boldsymbol{\nabla}_{\boldsymbol{x}} \bbu^\varepsilon \right\} + \mathcal{O}(\varepsilon^2).
\end{align}
This is an approximation of \eqref{eq:target_convection_diffusion} up to the first-order in $\varepsilon$ if we can ensure that
\begin{align}
	\varepsilon \frac{\ell}{||\Lambda||} \tilde{\mathbf{C}}^{-1} (\mathbf{J}_{\Lambda} - \mathbf{J}_{\bbf}) = \mathbf{D}.
\end{align}
Assuming that the matrix $(\mathbf{J}_\Lambda - \mathbf{J}_\bbf)$ is invertible, this leads to
\begin{align}
	\label{eq:relation_C_D}
	\varepsilon \frac{\ell}{||\Lambda||} \tilde{\mathbf{C}}^{-1} = \mathbf{D} \left( \mathbf{J}_\Lambda - \mathbf{J}_\bbf \right)^{-1},
\end{align}
and the inverse of the collision matrix appearing in \eqref{eq:kinetic_with_Omega} ($||\Lambda||/\ell \, \Omega/\varepsilon$) can be computed using \eqref{eq:Omega_C}. Since $\mathbf{D}$ is not assumed to be invertible (it is for example not invertible for the Navier-Stokes equations), \eqref{eq:relation_C_D} cannot be inverted to compute the relaxation matrix $\Omega$. But we will see in Sec.~\ref{sec:discretizations_IMEX} that this is not a problem. Thanks to the use of specific implicit time integrations for the collision term, a numerical method can be designed without the need for inverting $\mathbf{D}$. Hence, as in the one-dimensional case~\cite{wissocq2023Kinetic}, the convection-diffusion problem \eqref{eq:target_convection_diffusion} can be approximated by the kinetic model \eqref{eq:kinetic_with_Omega} under two assumptions:
\begin{itemize}
	\item[(i)] the matrix $\mathbf{J}_\Lambda - \mathbf{J}_\bbf$ is invertible,
	\item[(ii)] the second-order consistency error in $\varepsilon$ can be neglected, \textit{i.e.} $\varepsilon \ll 1$.
\end{itemize}

{\color{black}Regarding assumption (i), it is shown in \cite{Bouchut} that, for a system admitting entropies, having monotone entropies implies that  the matrix 
$$\big (\bA_0\otimes\mathbf{I}_d\big ) \big (\mathbf{J}_\Lambda - \mathbf{J}_\bbf\big )$$ 
is symmetric and positive definite if the eigenvalues of $\M'$ are in $[0,+\infty[$. This is theorem 2.7 of \cite{Bouchut}: the conclusion is that $\mathbf{J}_\Lambda - \mathbf{J}_\bbf$ is invertible as soon as the Maxwelians are monotone. It also is easy to show that with the classical wave models D2Q(2n) and D3Q(3n), the matrix $\mathbf{J}_\Lambda$ is diagonal. }

\begin{example}[Four-wave model]
	Following the notations of Example~\ref{ex:moments_matrix_D2Q4} for the kinetic speeds of the four-wave model, we have $\mathbf{J}_\Lambda = a^2 \mathbf{I}_p$. Furthermore, if the inviscid case of \eqref{eq:target_convection_diffusion} is symmetrizable, then the matrix $\mathbf{J}_\bbf$ is also symmetrizable. Hence, if $a$ is large enough, the matrix $\mathbf{J}_\Lambda - \mathbf{J}_\bbf = a^2 \mathbf{I}_p - \mathbf{J}_\bbf$ is invertible.
\end{example}

Regarding assumption (ii), it is paramount to have a correct estimation of the Knudsen number in order to check if it can be considered reasonably small. Assuming that $||\tilde{\mathbf{C}}||=\mathcal{O}(1)$, \eqref{eq:relation_C_D} yields
\begin{align}
	\varepsilon \approx \frac{||\Lambda||}{\ell} \left \Vert \mathbf{D}(\mathbf{J}_\Lambda - \mathbf{J}_\bbf)^{-1} \right \Vert.
\end{align}
Noticing that $\mathbf{J}_\Lambda$ is proportional to $||\Lambda||^2$, we can observe that
\begin{align}
	\label{eq:def_kn}
	\varepsilon \approx \frac{||\mathbf{D}||}{||\Lambda|| \ell},
\end{align}
which is the general definition of Knudsen number we will adopt in the following. As in~\cite{wissocq2023Kinetic}, note that $\varepsilon$ is inversely proportional to $||\Lambda||$, which allows us to arbitrarily decrease the Knudsen number, hence the consistency error, by increasing the norm of the kinetic speeds.

A last discussion can be raised regarding the role of the parameter $\alpha$ in the collision matrix $\mathbf{C}$ written in the moment space \eqref{eq:def_C_alpha}. We see that $\alpha$ never appears in the Chapman-Enskog expansion of this section, nor in the condition \eqref{eq:relation_C_D} that should be satisfied to approach \eqref{eq:target_convection_diffusion} up to the second-order in $\varepsilon$. This feature can be explained by the different behavior of the moments computed by the matrix $\momentsmatrix$. According to \eqref{eq:def_C_alpha}, see that the relaxation parameter $\alpha$ is applied to two types of moments: the first-order moments (obtained by application of the projector $\P$) and the high-order ones (obtained by application of the matrix $\mathbb{H}$). Regarding the first-order moments, since they are, by definition, collision invariants, their relaxation parameter has absolutely no influence on the conservation equation. This property appears when \eqref{eq:conservation_uepsilon} has been obtained. Regarding the high-order moments, their relaxation parameter is expected to have an impact on the conservation equation of $\bbu^\varepsilon$, but only in the $\mathcal{O}(\varepsilon^2)$ related terms. This is for example demonstrated in~\cite{SHAN2006}, where the \emph{moments cascade}, \textit{i.e.} the way high-order moments can affect high-order terms in $\varepsilon$, is exhibited. Note that the main principle of this moments cascade is very general and does not depend on the adopted definition of moments. Consequently, since, in the proposed work, we want to target the convection-diffusion problem \eqref{eq:target_convection_diffusion} up to the second-order in $\varepsilon$, the parameter $\alpha$ is free in our kinetic model. Also note that it only appears by its inverse $1/\alpha$ since only the inverse of the collision matrix is involved in the numerical methods of Sec.~\ref{sec:discretizations_IMEX}. Notably, left-multiplying \eqref{eq:Chapman_F} by $\mathbb{H}$, we obtain
\begin{align}
\label{eq:high-order_moments_equation}
	\mathbb{H} \bF = \mathbb{H} \M (\P \bF) - \varepsilon \frac{\ell}{||\Lambda||} \frac{1}{\alpha} \left( \dpar{(\mathbb{H} \bF)}{t} + \sum_{j=1}^d \mathbb{H} \Lambda_j \dpar{\bF}{x_j} \right),
\end{align}
which exhibits how the high-order moments can be affected by the parameter $1/\alpha$. In the following, we will consider the limit case $1/\alpha=0$, which will be referred to as a \emph{regularized} kinetic models.

\begin{definition}[Regularized kinetic models]
	A kinetic model with a collision matrix as \eqref{eq:kinetic_with_Omega} is referred to as a regularized model when $\Omega$ is such that
	\begin{align}
		\Omega^{-1} = \momentsmatrix^{-1} 
	\begin{pmatrix}
		\mathbf{0}_{p\times p} & \mathbf{0}_{p \times dp} & \mathbf{0}_{p\times hp} \\
		\mathbf{0}_{dp \times p} & \tilde{\mathbf{C}}^{-1} & \mathbf{0}_{dp \times hp} \\
		\mathbf{0}_{hp \times p} & \mathbf{0}_{hp \times dp} & \mathbf{0}_{hp \times hp}
	\end{pmatrix}
 	\momentsmatrix,
	\end{align}
	where $\momentsmatrix$ is a collision matrix and $\tilde{\mathbf{C}}^{-1}$ is related to $\mathbf{D}$ through \eqref{eq:relation_C_D}.
\end{definition}

Following \eqref{eq:high-order_moments_equation}, a regularized model exactly yields $\mathbb{H}\bF = \mathbb{H} \M (\P \bF)$, meaning that the high-order moments are immediately relaxed to the Maxwellian state. The denomination \emph{regularized model}  originates from the lattice Boltzmann  framework, where such a modeling of the collision term is sometimes considered~\cite{Latt2006, Malaspinas2015, Mattila2017, Coreixas2020}. Note that in the LBM, these models are commonly adopted in order to increase the numerical stability of the method~\cite{Coreixas2020, Wissocq2020}. In the proposed kinetic model, this choice is only made to increase the efficiency of the numerical method by reducing the number of operations and the memory cost, as will be shown in the next section. Interestingly, it is this reason that motivated the first use of regularized models in the LBM community, to the best of our knowledge~\cite{Skordos1993, Ladd2001}.

\subsection{Regularized collision matrix models as Jin-Xin methods}
\label{sec:regul_as_jinxin}

The above discussion introduces the concept of \emph{regularization} of kinetic methods in presence of a collision matrix, leading to $\mathbb{H} \mathbf{F} = \mathbb{H} \M(\P \bF)$, \textit{i.e.} high-order moments are always at the Maxwellian state. Moreover, when $\mathbb{H}$ is built such that $\mathbb{H} \M (\P \bF) = \mathbf{0}$, this leads to $\mathbb{H} \bF = \mathbf{0}$. In this section, we show how this strategy can be used to reduce the rank of the kinetic model from $kp$ to $(d+1)k$, whatever the wave model under consideration. This observation can be used to reduce the memory cost of the kinetic scheme. Furthermore, we show how connections can be established with Jin-Xin methods~\cite{Jin}, where the fluxes of a conservation equation obey an advection-relaxation equation.

\begin{proposition}\label{prop:1}
	Denoting $\bbu^\varepsilon = \P \bF$ and $\bbv_i^\varepsilon = \P \Lambda_i \bF$, the regularized collision matrix model \eqref{eq:kinetic_with_Omega} is equivalent to the following Jin-Xin system
\begin{align}
\label{eq:regul_as_JinXin}
	\dpar{}{t} 
	\begin{pmatrix}
		\bbu^\varepsilon \\ \bbv_1^\varepsilon \\ \vdots \\ \bbv_d^\varepsilon
	\end{pmatrix}
	+ \sum_{i=1}^d  \mathbf{A}_i \dpar{}{x_i}
	\begin{pmatrix}
		{\bbu_i^\varepsilon} \\ \bbv_1^\varepsilon \\ \vdots \\ \bbv_d^\varepsilon
	\end{pmatrix} =
	\begin{bmatrix}
		\mathbf{0} \\ 
		\dfrac{\Vert\Lambda\Vert}{\ell} \dfrac{\tilde{\mathbf{C}}}{\varepsilon}
		\begin{pmatrix}
			\bbf_1(\bbu^\varepsilon) - \bbv_1^\varepsilon \\
			\vdots \\
			\bbf_d(\bbu^\varepsilon) - \bbv_d^\varepsilon
		\end{pmatrix}
	\end{bmatrix},
\end{align}
where 
  $\mathbf{A}_i= \overline{\momentsmatrix} \Lambda_i \overline{\momentsmatrix}^+$, $\overline{\momentsmatrix}$ is a reduced moments matrix defined as
\begin{align}
	\overline{\momentsmatrix} = 
	\begin{pmatrix}
		\P \\ \P \Lambda_1 \\ \vdots \\ \P \Lambda_d
	\end{pmatrix} \in \mathcal{M}_{(d+1)p \times kp}(\mathbb{R}),
\end{align}
and $\overline{\momentsmatrix}^+ \in \mathcal{M}_{kp\times (d+1)p} (\mathbb{R})$ is its pseudo inverse, i.e. the only matrix satisfying $\overline{\momentsmatrix} \overline{\momentsmatrix}^+ \overline{\momentsmatrix} = \overline{\momentsmatrix}$ and $\overline{\momentsmatrix}^+ \overline{\momentsmatrix} \overline{\momentsmatrix}^+ = \overline{\momentsmatrix}^+$.
Eq.~\eqref{eq:regul_as_JinXin} is a Jin-Xin system involving the collision matrix in the subspace of the first-order moments $\tilde{\mathbf{C}}$. Only the matrices $\mathbf{A}_i$ and $\tilde{\mathbf{C}}$ depend on the considered wave model.
\end{proposition}
\begin{proof}
	Left-multiplying \eqref{eq:kinetic_with_Omega} by $\overline{\momentsmatrix}$ leads to
	\begin{align}
		\dpar{}{t}
		\begin{pmatrix}
			\bbu^\varepsilon \\ \bbv_1^\varepsilon \\ \vdots \\ \bbv_d^\varepsilon
		\end{pmatrix} + \sum_{i=1}^d
		\overline{\momentsmatrix} \Lambda_i
		\dpar{\bF}{x_i} = 
		\begin{bmatrix}
			\mathbf{0} \\
			\dfrac{||\Lambda||}{\ell}
			\begin{pmatrix}
				\P \Lambda_1 \\ \vdots \\\P \Lambda_d
			\end{pmatrix}
			\dfrac{\Omega}{\varepsilon}(\M(\P \bF) - \bF)
		\end{bmatrix}.
		\label{eq:demo_jinxin_1}
	\end{align}
	Using \eqref{eq:Omega_C_permutation} with $\Omega$ instead of $\Omega^{-1}$ and the fact that $\P \Lambda_i \M (\P \bF) = \bbf_i(\bbu^\varepsilon)$ yields
	\begin{align}
		\begin{pmatrix}
			\P \Lambda_1 \\ \vdots \\ \P \Lambda_d 
		\end{pmatrix}
		\frac{\Omega}{\varepsilon} (\M (\P \bF) - \bF) = \frac{\tilde{\mathbf{C}}}{\varepsilon} 
		\begin{pmatrix}
			\bbf_1(\bbu^\varepsilon) - \bbv_1^\varepsilon \\ 
			\vdots \\
			\bbf_d(\bbu^\varepsilon) - \bbv_d^\varepsilon
		\end{pmatrix}, 
	\end{align}
	so that the right-hand-side term of Eq.~\eqref{eq:regul_as_JinXin} can be obtained. To obtain a Jin-Xin like model, one needs to close the system with the variables $(\bbu^\varepsilon, \bbv_i^\varepsilon)$ only, while $\bF$ still appears in the transport term of \eqref{eq:demo_jinxin_1}. To deal with this term, we can notice that
\begin{align}
	\bF = \momentsmatrix^{-1} \mathbf{m}^\bF, \qquad \mathbf{m}^\bF = [\bbu^\varepsilon,
		\bbv_1^\varepsilon, \hdots,\bbv_d^\varepsilon,
		\mathbf{0}]^T,
\end{align}
which is a function of $(\bbu^\varepsilon, \bbv_i^\varepsilon)$ only, where we used the fact that $\mathbb{H} \bF = \mathbb{H} \M (\bbu^\varepsilon) = \mathbf{0}$ (regularization). Using a chain rule, 
\begin{align}
	\dpar{\bF}{x_i} = \momentsmatrix^{-1} \dpar{\mathbf{m}^\bF}{x_i} = 
	\momentsmatrix^{-1}
	\begin{pmatrix}
		\mathbf{I}_p & \mathbf{0}_{p \times dp} \\
		\mathbf{0}_{dp \times p} & \mathbf{I}_{dp} \\
		\mathbf{0}_{hp \times p} & \mathbf{0}_{hp \times dp}
	\end{pmatrix}
	\dpar{}{x_i}
	\begin{pmatrix}
		\bbu^\varepsilon \\ \bbv_1^\varepsilon \\ \vdots \\ \bbv_d^\varepsilon
	\end{pmatrix}.
\end{align}
We finally show that the matrix 
\begin{align}
	\overline{\momentsmatrix}^+ := \momentsmatrix^{-1} 
	\begin{pmatrix}
		\mathbf{I}_p & \mathbf{0}_{p \times dp} \\
		\mathbf{0}_{dp \times p} & \mathbf{I}_{dp} \\
		\mathbf{0}_{hp \times p} & \mathbf{0}_{hp \times dp}
	\end{pmatrix}
\end{align}
is the pseudo-inverse of $\overline{\momentsmatrix}$ to conclude the demonstration. Noticing that
\begin{align}
	\overline{\momentsmatrix} = 
	\begin{pmatrix}
		\mathbf{I}_p & \mathbf{0}_{p \times dp} & \mathbf{0}_{p \times hp} \\
		\mathbf{0}_{dp \times p} & \mathbf{I}_{dp} & \mathbf{0}_{dp \times hp}
	\end{pmatrix} \momentsmatrix,
\end{align}
we have the direct result $\overline{\momentsmatrix} \overline{\momentsmatrix}^+ = \mathbf{I}_{(d+1)p}$, which implies $\overline{\momentsmatrix} \overline{\momentsmatrix}^+ \overline{\momentsmatrix} = \overline{\momentsmatrix}$ and $\overline{\momentsmatrix}^+ \overline{\momentsmatrix} \overline{\momentsmatrix}^+ = \overline{\momentsmatrix}^+$.
\end{proof}

\begin{example}[Four-wave model in two dimensions]
\label{ex:4waves_2D_JinXin}
	Considering the matrices $\Lambda_1$, $\Lambda_2$ and $\momentsmatrix$ of Example \ref{ex:moments_matrix_D2Q4}, we have
	\begin{align}
		\overline{\momentsmatrix} = 
		\begin{pmatrix}
			1 & 1 & 1 & 1 \\
			-a & a & 0 & 0 \\
			0 & 0 & -a & a
		\end{pmatrix} \otimes \mathbf{I}_p, \qquad
		\overline{\momentsmatrix}^+ = 
		\begin{pmatrix}
			1/4 & -1/(2a) & 0 \\
			1/4 & 1/(2a) & 0 \\
			1/4 & 0 & -1/(2a) \\
			1/4 & 0 & 1/(2a)
		\end{pmatrix} \otimes \mathbf{I}_p,
	\end{align}
	leading to
	\begin{align}
		\mathbf{A}_1 = 
		\begin{pmatrix}
			0 & 1 & 0 \\
			a^2/2 & 0 & 0 \\
			0 & 0 & 0
		\end{pmatrix} \otimes \mathbf{I}_p, \qquad
		\mathbf{A}_2 = 
		\begin{pmatrix}
			0 & 0 & 1 \\
			0 & 0 & 0 \\
			a^2/2 & 0 & 0
		\end{pmatrix}.
	\end{align}
	This simply reads
	\begin{align}
		& \dpar{\bbu^\varepsilon}{t} + \dpar{\bbv_1^\varepsilon}{x_1} + \dpar{\bbv_2^\varepsilon}{x_2} = \mathbf{0}, \\
		& \dpar{}{t}
		\begin{pmatrix}
			\bbv_1^\varepsilon \\ \bbv_2^\varepsilon
		\end{pmatrix} + \frac{a^2}{2} \dpar{}{x_1}
		\begin{pmatrix}
			\bbu^\varepsilon \\ \mathbf{0}
		\end{pmatrix}
		+ \frac{a^2}{2} \dpar{}{x_2}
		\begin{pmatrix}
			\mathbf{0} \\ \bbu^\varepsilon
		\end{pmatrix} = \frac{||\Lambda||}{\ell} \frac{\tilde{\mathbf{C}}}{\varepsilon}
		\begin{pmatrix}
			\bbf_1(\bbu^\varepsilon) - \bbv_1^\varepsilon \\
			\bbf_2(\bbu^\varepsilon) - \bbv_2^\varepsilon
		\end{pmatrix}.
	\end{align}
	This is exactly the two-dimensional model detailed in~\cite{Jin}, except that a collision matrix is introduced in the relaxation of $(\bbv_1^\varepsilon, \bbv_2^\varepsilon)$.
\end{example}

\begin{example}[Eight-velocity model in two dimensions]
	Considering the regularized kinetic model described in Example~\ref{ex:moments_matrix_D2Q8}, we have
	\begin{align}
		\overline{\momentsmatrix} = 		\begin{pmatrix}
			1 & 1 & 1 & 1 & 1 & 1 & 1 & 1 \\
			a & a & 0 & -a & -a & -a & 0 & a \\
			0 & a & a & a & 0 & -a & -a & -a
		\end{pmatrix} \otimes \mathbf{I}_p , \qquad 
		\overline{\momentsmatrix}^+ = 
		\begin{pmatrix}
			1/8 & 1/(6a) & 0 \\
			1/8 & 1/(6a) & 1/(6a) \\
			1/8 & 0 & 1/(6a) \\
			1/8 & -1/(6a) & 1/(6a) \\
			1/8 & -1/(6a) & 0 \\
			1/8 & -1/(6a) & -1/(6a) \\
			1/8 & 0 & -1/(6a) \\
			1/8 & 1/(6a) & -1/(6a)
		\end{pmatrix},
	\end{align}
	leading to
	\begin{align}
		\mathbf{A}_1 = 
		\begin{pmatrix}
			0 & 1 & 0 \\
			3a^2/4 & 0 & 0 \\
			0 & 0 & 0 
		\end{pmatrix} \otimes \mathbf{I}_p, \qquad
		\mathbf{A}_2 = 
		\begin{pmatrix}
			0 & 0 & 1 \\
			0 & 0 & 0 \\
			3a^2/4 & 0 & 0
		\end{pmatrix} \otimes \mathbf{I}_p.
	\end{align}
\end{example}

\begin{remark}
	Note that the eigenvalues of the hyperbolic (left-hand-side) term of the Jin-Xin model~\eqref{eq:regul_as_JinXin} are in general different from those of the kinetic model~\eqref{eq:kinetic_with_Omega}. Notably with the Example~\ref{ex:4waves_2D_JinXin}, the eigenvalues of the Jin-Xin advection term are: $(0, \pm a/\sqrt{2})$, while the eigenvalues of the kinetic system are: $(0, \pm a)$. 
\end{remark}

\section{Numerical discretization: arbitrarily order IMEX methods}
\label{sec:discretizations_IMEX}

In this section, we recall the time and space discretizations proposed in previous work~\cite{Torlo, Abgrall2023, wissocq2023Kinetic} to  build robust numerical methods for approximating \eqref{eq:kinetic_with_Omega}. We start by assuming a semi-discrete form of the kinetic system \eqref{eq:kinetic_with_Omega} involving a discrete operator $\delta_{x_i}^q$, that approximates the space derivatives $\partial_{x_i}$ with a $q$th-order of accuracy. This reads
\begin{align}
\label{eq:kinetic_system_omega_semidiscrete}
    \frac{\mathrm{d} \bF}{\mathrm{d} t} = - \sum_{i=1}^d \Lambda_i \delta_{x_i}^q \bF + \frac{||\Lambda||}{\ell} \frac{\Omega}{\varepsilon} (\P \bF) \left( \M(\P \bF) - \bF \right) := \mathcal{F}(\bF).
\end{align}
We first recall the main principle and properties of the arbitrary-order schemes adopted in previous work to integrate this equation in time~\cite{Torlo, Abgrall2023, wissocq2023Kinetic}: an implicit-explicit first-order scheme and high-order schemes based on the deferred correction (DeC) algorithm. We then discuss the space discretization to build robust methods of arbitrary order in space and time, the numerical stability of the ensuing numerical methods and the introduction of boundary conditions.

\subsection{First-order explicit scheme in time}
\label{sec:first_order_time}

Following~\cite{Torlo, Abgrall2023, wissocq2023Kinetic}, we use a first-order explicit integration for the convective part and a first-order implicit integration for the stiff collision term. Integrating between $t_n$ and $t_{n+1} := t_n + \Delta t$, this reads
\begin{align}
\label{eq:first-order_implicit}
    \bF (t_{n+1}) - \bF(t_n) = - \Delta t \sum_{i=1}^d \Lambda_i \delta_{x_i}^q \bF (t_n) + \Delta t \, \frac{||\Lambda||}{\ell} \frac{\Omega }{\varepsilon}(\P \bF (t_{n+1})) \left[ \M(\P \bF (t_{n+1})) - \bF(t_{n+1}) \right].
\end{align}
This scheme is implicit, but can be made fully explicit by first applying the projector $\P$ to the solution at time $t_{n+1}$, leading to
\begin{align}
\label{eq:first-order_IMEX_PF}
    \P \bF(t_{n+1}) = \P \bF (t_n) - \Delta t \sum_{i=1}^d \P \Lambda_i \delta_{x_i}^q \bF(t_n).
\end{align}
Then defining
\begin{align}
	\hat{\Omega}_{n+1} = \Delta t \frac{||\Lambda||}{\ell} \frac{\Omega}{\varepsilon} (\P \bF(t_{n+1})),
\end{align}
$\bF(t_{n+1})$ can be explicitly computed  by reversing the implicitness in \eqref{eq:first-order_implicit}, leading to
\begin{align}
\label{eq:first_order_IMEX_F}
    \bF(t_{n+1}) = \left[ \mathbf{I}_{kp} + \hat{\Omega}_{n+1}^{-1} \right]^{-1} \left\{ \hat{\Omega}_{n+1}^{-1} \left[ \bF(t_n) -\Delta t \sum_{i=1}^d \Lambda_i \delta_{x_i}^q \bF(t_n) \right] + \M (\P \bF(t_{n+1})) \right\}.
\end{align}
As in~\cite{wissocq2023Kinetic}, this is an explicit scheme involving the inverse of the collision matrix, which can be computed even when $\mathbf{D}$ is not invertible thanks to \eqref{eq:relation_C_D}.

\subsection{High-order scheme in time: deferred correction}
\label{sec:DeC}

In order to build high-order discretizations in time of \eqref{eq:kinetic_system_omega_semidiscrete}, we adopt a DeC algorithm. For more details regarding DeC applied to kinetic models, the interested reader is referred to~\cite{Torlo, Abgrall2023, wissocq2023Kinetic}. In this section, we briefly recall the main principle of the DeC and detail the explicit iterative schemes obtained when discretizing \eqref{eq:kinetic_system_omega_semidiscrete} at arbitrary order in time.

Consider a time integration of \eqref{eq:kinetic_system_omega_semidiscrete} using an implicit Runge-Kutta scheme of order $q$ involving $s$ sub-time nodes denoted as $c_1 = 0 < c_2 < .. < c_s = 1$. We note $\mathbf{F}_i$, for $i \in \{1, \dots, s\}$, the solution corresponding to the sub-time node $c_i$ and define the vector $\hat{\bF} = [\bF_1, \dots, \bF_s]^T$ of size $kps$. Without loss of generality, the RK scheme can be written as $\mathcal{L}^2(\hat{\bF}) = 0$, where $\mathcal{L}^2$ is a ``high-order'' operator. Noticing that this problem may be difficult to solve, especially because it is implicit and non-linear, the purpose of the DeC is to introduce a new ``low-order'' operator $\mathcal{L}^1$ such that $\mathcal{L}^1(\hat{\bF}) = 0$ yields an explicit scheme to solve \eqref{eq:kinetic_system_omega_semidiscrete} with a lower order of accuracy. In practice, we choose $\mathcal{L}^1$ such that the integration of the stiff relaxation term is the same implicit scheme as in $\mathcal{L}^2$, while the transport term is integrated by an explicit scheme. As shown in Sec.~\ref{sec:first_order_time}, the implicitness of the relaxation term is not a problem since it vanishes when applying the projector $\P$. Then, as shown in \cite{Torlo}, the algorithm 
\begin{enumerate}
\item $\hat{\bF}^{(0)}=[\bF(t_n), \dots, \bF(t_n)]^T$,
\item $\hat{\bF}^{(p+1)}$ solution of $\mathcal{L}^1(\hat{\bF}^{(p+1)})=\mathcal{L}^1(\hat{\bF}^{(p)})-\mathcal{L}^2(\hat{\bF}^{(p)})$,
\end{enumerate}
is such that $\hat{\bF}^{(q)}-\hat{\bF}(t_{n+1})=O(\Delta t^q)$ where $\hat{\bF}(t_{n+1})$ is the solution of $\mathcal{L}_2(\hat{\bF})=0$ at time $t_{n+1}=t_n+\Delta t$. As detailed in~\cite{wissocq2023Kinetic}, when applied to a kinetic system such as \eqref{eq:kinetic_system_omega_semidiscrete}, the DeC iteration reads,
\begin{align}
	\forall p \in \{0, \dots, q-1\}, \quad &     \P \hat{\bF}^{(p+1)} = \P \hat{\bF}^{(0)}  - (\mathbf{I}_s \otimes \P) \hat{\mathbf{A}} \sum_{i=1}^d \hat{\Lambda}_i \delta_{x_i}^q \hat{\bF}^{(p)},
    \label{eq:DeC_macros} \\
    &     \hat{\bF}^{(p+1)} = \hat{\Omega}_{(p+1)}^{-1} \left[ \hat{\Omega}_{(p+1)}^{-1} + \hat{\mathbf{A}} \right]^{-1} \left( \hat{\bF}^{(0)} - \hat{\mathbf{A}} \sum_{i=1}^d \hat{\Lambda}_i \delta_{x_i}^q \hat{\bF}^{(p)} \right) + \hat{\mathbf{A}} \left[ \hat{\Omega}_{(p+1)}^{-1} + \hat{\mathbf{A}} \right]^{-1} \hat{\M}^{(p+1)},
    \label{eq:DeC_explicit}
\end{align}
where the following vectors and matrices have been defined,
\begin{align}
	\hat{\M}^{(p)} = 
	\begin{pmatrix}
		\M(\P \bF_1^{(p)}) \\ \vdots \\ \M(\P \bF_s^{(p)})
	\end{pmatrix},  \quad 
	\hat{\Omega}_{(p)} = \frac{||\Lambda||}{\ell}    
	\begin{pmatrix}
        \Omega/\varepsilon (\P \bF_1^{(p)}) & \hdots & 0 \\
        \vdots & \ddots & \vdots \\
        0 & \hdots & \Omega/\varepsilon (\P \bF_s^{(p)})
    \end{pmatrix}, \quad \hat{\Lambda}_i = \mathbf{I}_s \otimes \Lambda_i, \quad \hat{\mathbf{A}} = \Delta t \mathbf{A} \otimes \mathbf{I}_{kp},
\end{align}
and where $\mathbf{A}$ is the matrix of the coefficients of the implicit RK algorithm. The iterative procedure \eqref{eq:DeC_macros}-\eqref{eq:DeC_explicit} is an explicit scheme allowing us to compute $\hat{\bF}^{(q)}$. Note that, as with the explicit scheme of Sec.~\ref{sec:first_order_time}, it only involves $\Omega$ through its inverse matrix, which can be computed even when $\mathbf{D}$ is not invertible as shown by \eqref{eq:relation_C_D}. Finally, the approximated solution at time $t_{n+1}$ can be obtained in a general way as
\begin{align}
	\bF(t_{n+1}) = \bF(t_n) + \Delta t \sum_{k=1}^s b_k \mathcal{F}(\bF_k^{(q)}),
\end{align} 
where $b_k$ are coefficients of the considered RK scheme. In the present work, we will only consider Lobato IIIC schemes, for which the computation of the updated solution is simply done by $\bF(t_{n+1}) = \bF_s^{(q)}$. The matrices $\mathbf{A}$ of the second- and fourth-order Lobato IIIC schemes are recalled below, from~\cite{Hairer},
\begin{align}
	& \mathrm{Second-order}: \qquad \mathbf{A} = 
    \begin{pmatrix}
        1/2 & -1/2 \\
        1/2 & 1/2
    \end{pmatrix},
    \label{eq:LobatoIIIC_O2} \\
    & \mathrm{Fourth-order}: \qquad 
    \mathbf{A} = 
    \begin{pmatrix}
        1/6 & -1/3 & 1/6 \\
        1/6 & 5/12 & -1/12 \\
        1/6 & 2/3 & 1/6
    \end{pmatrix}.
    \label{eq:LobatoIIIC_O4}
\end{align}

\subsection{Space discretizations}
\label{sec:space_disc}

We now discuss the space discretization of the term $\Lambda_i \partial_{x_i}$ with a $q$th-order scheme in space referred to as $\Lambda_i \delta_{x_i}^q$. Without loss of generality, we focus on the one-dimensional case and drop the index $i$ for the space dimension. We consider a uniform space discretization with equally spaced points $x_k$ with a uniform mesh size $\Delta x=x_{k+1}-x_k$. We note $F_{j,k}$ the $j$th component of $\bF$ at point $x_k$ and $\Lambda = \mathrm{diag}(a_j, j \in \{1,\dots kp\})$. In order to systematically ensure the conservativity of the method, we adopt a finite-volume formalism and write the discrete transport term as
\begin{align}
	a_j \delta_x^q F_{j,k} = \frac{\Phi_{j,k+1/2}^q - \Phi_{j,k-1/2}^q}{\Delta x},
\end{align}
where $\Phi_{j,k+1/2}^q$ is an estimation of the flux at order $q$ between the cells centered in $x_k$ and $x_{k+1}$, related to the kinetic speed $a_j$. The fluxes adopted in this work, as previously in~\cite{AbgrallK} and inspired from \cite{iserle}, are provided below.

\paragraph{First-order ($\Phi^1$)} We use the upwind scheme,
\begin{align}
	\Phi_{j,k+1/2}^1 = 
	\begin{cases}
		a_j F_{j,k} \qquad \mathrm{if}\ a_j \geq 0, \\
		a_j F_{j,k+1} \qquad \mathrm{else}.
	\end{cases}
\end{align}

\paragraph{Second-order ($\Phi^2$)} We define
\begin{align}
	\Phi_{j,k+1/2}^2 = 
	\begin{cases}
		a_j \left( \frac{1}{3} F_{j,k+1} + \frac{5}{6} F_{j,k} - \frac{1}{6} F_{j,k-1} \right) \qquad \mathrm{if}\ a_j \geq 0, \\
		a_j \left( -\frac{1}{6} F_{j,k+2} + \frac{5}{6} F_{j,k+1} + \frac{1}{3} F_{j,k} \right) \qquad \mathrm{else}.
	\end{cases}
	\end{align}

\paragraph{Fourth-order ($\Phi^4$)} We define
\begin{align}
	\Phi_{j,k+1/2} = 
	\begin{cases}
		a_j \left( \frac{1}{4} F_{j,k+1} + \frac{13}{12} F_{j,k} - \frac{5}{12} F_{j,k-1} + \frac{1}{12} F_{j,k-2} \right) \qquad \mathrm{if}\ a_j \geq 0, \\
		a_j \left( \frac{1}{12} F_{j,k+3} - \frac{5}{12} F_{j,k+2} + \frac{13}{12} F_{j,k+1} + \frac{1}{4} F_{j,k} \right) \qquad \mathrm{else}.
	\end{cases}
\end{align}

Note that the fourth-order scheme is different from the one adopted in~\cite{wissocq2023Kinetic}, which was a central scheme. A linear stability analysis of the advection term discretized in space with $\delta_x^q$ and in time with either the first-order scheme of Sec.~\ref{sec:first_order_time}, or the DeC algorithm of Sec.~\ref{sec:DeC}, can be performed in the same way as in \cite{wissocq2023Kinetic}. Table \ref{tab:max_CFL_LobatoIIIC} provides the critical values of $\lambda = a \Delta t/\Delta x$ ensuring linear stability in each case, and for different numbers of iterations when the DeC algorithm is concerned.

\begin{table}[!ht]
  \centering
\begin{tabular}{|p{0.15\textwidth}|>{\centering}p{0.05\textwidth}||>{\centering}p{0.05\textwidth}|>{\centering}p{0.05\textwidth}|>{\centering}p{0.05\textwidth}|>{\centering}p{0.05\textwidth}|>{\centering}p{0.05\textwidth}|>{\centering\arraybackslash}p{0.05\textwidth}|}
\hline
\multicolumn{2}{|c|}{Scheme} & \multicolumn{6}{c|}{\# iterations} \\
\hline
Order in time & $\delta_x^q$ & 1 & 2 & 3 & 4 & 5 & 6 \\
\hline
1 & $\delta_x^1$ & 1 & - & - & - & - & - \\
1 & $\delta_x^2$ & 0 & - & - & - & - & - \\
1 & $\delta_x^4$ & 0 & - & - & - & - & - \\
\hline
2 (DeC) & $\delta_x^1$ & 1 & 1 & 1 & 0.78 & 0.71 & 0.85 \\
2 (DeC) & $\delta_x^2$ & 0 & 0.87 & 0.87 & 0.96 & 0.88 & 0.98 \\
2 (DeC) & $\delta_x^4$ & 0 & 0.12 & 0.12 & 0.54 & 0.53 & 0.61 \\
\hline
4 (DeC) & $\delta_x^1$ & 1 & 1 & 1.26 & 1.39 & 1.46 & 1.34\\
4 (DeC) & $\delta_x^2$ & 0 & 0.87 & 1.63 & 1.75 & 1.81 & 1.77\\
4 (DeC) & $\delta_x^4$ & 0 & 0.12 & 0.90 & 1.04 & 1.09 & 1.00 \\
\hline
\end{tabular}
\caption{Critical values of $\lambda = a \Delta t/\Delta x$ ensuring linear stability of different combinations of time and space discretizations of the transport term in~\eqref{eq:kinetic_with_Omega}.}
\label{tab:max_CFL_LobatoIIIC}
\end{table}

\subsection{Boundary conditions}
\label{sec:boundaries}

We now focus on the introduction of boundary conditions for the proposed arbitrarily high-order methods. The numerical schemes introduced in Sec.~\ref{sec:first_order_time} and \ref{sec:DeC} are purely local, except for the discrete convection term $\Lambda_i \delta_{x_i}^q$. Therefore, the consideration of boundary conditions only involves this non-local term. The question at stake in this section is: how to modify the space discretization of the proposed method close to a boundary such that (1) the scheme is consistent with the value we want to impose at the boundary, (2) the robustness of the numerical method is preserved, (3) the conservation of some useful quantity (for example mass close to a wall) can be systematically ensured?

\begin{figure}[h]
     \centering
     \begin{tikzpicture}[scale=0.8]
        \tikzstyle{point}=[fill,circle,scale=0.5]
        \tikzstyle{center}=[circle,draw,thick,fill=blue!45]
        \tikzstyle{used}=[draw,rectangle,thick,fill=white]
        \draw (0,0) grid[step=1.] (8,0);
        \draw (1,0) node[point]{};  
        \draw (3,0) node[point]{};  
        \draw (5,0) node[point]{};  
        \draw (7,0) node[point]{};
        \draw [thick] (0,-0.7) -- (0,0.7);  
        \draw [thick] (-0.2,-0.7) -- (0,-0.6);  
        \draw [thick] (-0.2,-0.5) -- (0,-0.4);  
        \draw [thick] (-0.2,-0.3) -- (0,-0.2);  
        \draw [thick] (-0.2,-0.1) -- (0,0.);  
        \draw [thick] (-0.2,0.1) -- (0,0.2);  
        \draw [thick] (-0.2,0.3) -- (0,0.4);  
        \draw [thick] (-0.2,0.5) -- (0,0.6);  
        \draw [dashed, thick] (2,-0.7) -- (2,0.7);  
        \draw [dashed, thick] (4,-0.7) -- (4,0.7);  
        \draw [dashed, thick] (6,-0.7) -- (6,0.7);
		\draw (0,-1.2) node {$\mathbf{\Phi}_b$}; 
		\draw (2,-1.2) node {$\mathbf{\Phi}^q_{1+1/2}$};
		\draw (4,-1.2) node {$\mathbf{\Phi}^q_{2+1/2}$};
		\draw (6,-1.2) node {$\mathbf{\Phi}^q_{3+1/2}$};
		\draw (1,-0.5) node {$x_1$};
		\draw (3,-0.5) node {$x_2$};
		\draw (5,-0.5) node {$x_3$};
		\draw (7,-0.5) node {$x_4$};
		\draw (0, 1.1) node {$\bbu_b$};
		\draw (1, 1.1) node {$\bF_1$};
		\draw (3, 1.1) node {$\bF_2$};
		\draw (5, 1.1) node {$\bF_3$};
		\draw (7, 1.1) node {$\bF_4$};
		 \end{tikzpicture}
     \caption{Sketch of the particular case of a left boundary considered with the finite volume formulation.}
    \label{fig:sketch_bnd}
\end{figure}
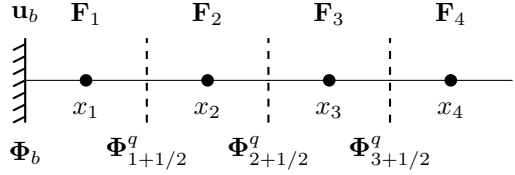

To fix the idea, we adopt the same notations as in Sec.~\ref{sec:space_disc} and focus on a particular case in one dimension displayed in Fig.~\ref{fig:sketch_bnd}. A left boundary is considered and we want to impose the conserved variables $\bbu_b$ at this boundary. As discussed in Sec.~\ref{sec:space_disc}, defining the space-derivative operator $\delta_x^q \bF$ is equivalent to defining a flux $\mathbf{\Phi}_{k+1/2}^q$ between cells centered about $x_k$ and $x_{k+1}$. Close to a boundary, new definitions of these fluxes have to be provided for two reasons:
\begin{enumerate}
	\item[(i)] At the boundary, the flux $\mathbf{\Phi}_b$ is undefined and has to be related to the value we want to impose at the boundary $\bbu_b$.
	\item[(ii)] Some of the fluxes in a small layer close to the boundary are not defined because, as seen in Sec.~\ref{sec:space_disc}, they involve neighbor points which may not exist. The thickness of this layer depends on the order of the spatial discretization.
\end{enumerate}

We first focus on point (i). By consistency with the kinetic system, the flux at the boundary can be written as $\mathbf{\Phi}_b = \Lambda_\bot \bF_b$, where $\Lambda_\bot$ is the matrix of the kinetic speeds in the direction normal to the boundary and $\bF_b$ is a set of distributions at the boundary, to be defined. We want to build $\bF_b$ such that $\P \bF_b = \bbu_b$ and recalling that $\P \Lambda_\bot \bF_b$ is the flux of $\bbu^\varepsilon$ normal to the boundary. It is paramount to properly impose these quantities. To this extent, we consider the particular case of a regularized kinetic model, in which $\bF_b$ can be related to the Jin-Xin variables $(\bbu^\varepsilon, \bbv_1^\varepsilon, \dots, \bbv_d^\varepsilon)$ through the pseudo-inverse of the reduced moments matrix $\overline{\momentsmatrix}^+$ as (\textit{cf.} Sec.~\ref{sec:regul_as_jinxin})
\begin{align}
	\bF_b = \overline{\momentsmatrix}^{+}
	\begin{pmatrix}
		\bbu_b \\ 
		\bbv_{1,b}^\varepsilon \\
		\vdots \\
		\bbv_{d,b}^\varepsilon
	\end{pmatrix}.
\end{align} 
With this choice, $\P \bF_b = \bbu_b$ is ensured by construction and $\P \Lambda_i \bF_b =: \bbv_{i, b}^\varepsilon$. The question left is the definition of the $\bbv_{i,b}^\varepsilon$, which are the fluxes of $\bbu^\varepsilon$ in the direction $x_i$ that we want to impose at the boundary. Here we set them to zero except for the one in the direction normal to the boundary $\bbv_{\bot, b}^\varepsilon$, that we define as
\begin{align}
	\bbv_{\bot, b}^\varepsilon = \bbf_\bot(\bbu_b) + \mathbf{d}_{\bot, b},
\end{align}
where $\bbf_\bot$ is the known convective flux in the direction normal to the boundary and $\mathbf{d}_{\bot, b}$ is an approximation of the diffusive flux, which remains to be defined. Remember that the diffusive flux is expected to be an approximation of $-\mathbf{D} \cdot \boldsymbol{\nabla}_{\boldsymbol{x}} \bbu^\varepsilon$, where $\mathbf{D}$ is the diffusion matrix. Therefore, one possibility to approximate $\mathbf{d}_{\bot, b}$ can be to rely on a finite-difference estimation of $\boldsymbol{\nabla}_{\boldsymbol{x}} \bbu^\varepsilon$. However, doing so, we would introduce a parabolic constraint on the time step $\Delta t= \mathcal{O}(\Delta x^2)$, which considerably reduces the efficiency and robustness of the overall scheme. To overcome this constraint, we compute $\mathbf{d}_{\bot, b}$ as a second-order interpolation of its value in the neighbor points close to the boundary as
\begin{align}
	\mathbf{d}_{\bot, b} = \frac{3}{2} \mathbf{d}_{\bot, 1} - \frac{1}{2} \mathbf{d}_{\bot, 2}, \qquad \mathbf{d}_{\bot, j} = \P \Lambda_\bot \bF_j - \bbf_\bot(\P \bF_j),
\end{align}
where $\bF_j$ stands for the distribution in the cell centered about $x_j$ (\textit{cf.} Fig.~\ref{fig:sketch_bnd}).

\begin{remark}
	Consider the mass density variable of the Navier-Stokes equations. Since it is not subject to diffusion, we have $\mathbf{d}_{\bot, j} = \mathbf{0}$ for this variable. As a consequence, $\P \Lambda_\bot \bF_b = \bbf_\bot (\bbu_b)$ which is equal to zero for a stationary wall. This shows that the proposed algorithm conserves mass by construction.
\end{remark}

Regarding point (ii), it can be addressed by locally switching the computation of the flux to a lower-order scheme, among those proposed in Sec.~\ref{sec:space_disc}. A possibility for doing so is proposed in Table~\ref{tab:bnds}, depending on the order in space and the sign of the kinetic velocity $a_j$. Note that the flux is null the case $a_j=0$. In the particular cases of second- and fourth-order schemes, a new second-order flux $\Phi_{j}^{2,*}$ is defined at point $x_{1+1/2}$ for a distribution $F_j$ advected at $a_j>0$ as
\begin{align}
\label{eq:flux_O2*}
	\Phi_{j,1+1/2}^{2,*} = a_j \left( F_{j,1} + \frac{1}{3}(F_{j,2}-F_{j,b}) \right).
\end{align}

\begin{table}[!ht]
  \centering
\begin{tabular}{|p{0.2\textwidth}||>{\centering}p{0.1\textwidth}|>{\centering}p{0.1\textwidth}|>{\centering}p{0.1\textwidth}||>{\centering}p{0.1\textwidth}|>{\centering}p{0.1\textwidth}|>{\centering\arraybackslash}p{0.1\textwidth}|}
\hline
& \multicolumn{3}{|c|}{$a_j>0$} & \multicolumn{3}{c|}{$a_j<0$} \\
\hline
Space discretization & $x_{1+1/2}$ & $x_{2+1/2}$ & $x_{3+1/2}$ & $x_{1+1/2}$ & $x_{2+1/2}$ & $x_{3+1/2}$ \\
\hline
First-order & $\Phi_{j,1+1/2}^1$ & $\Phi_{j,2+1/2}^1$ & $\Phi_{j,3+1/2}^1$ & $\Phi_{j,1+1/2}^1$ & $\Phi_{j,2+1/2}^1$ & $\Phi_{j,3+1/2}^1$ \\
Second-order & $\Phi_{j,1+1/2}^{2,*}$ & $\Phi_{j,2+1/2}^2$ & $\Phi_{j,3+1/2}^2$ & $\Phi_{j,1+1/2}^2$ & $\Phi_{j,2+1/2}^2$ & $\Phi_{j,3+1/2}^2$ \\
Fourth-order & $\Phi_{j,1+1/2}^{2,*}$ & $\Phi_{j,2+1/2}^2$ & $\Phi_{j,3+1/2}^4$ & $\Phi_{j,1+1/2}^4$ & $\Phi_{j,2+1/2}^4$ & $\Phi_{j,3+1/2}^4$ \\
\hline
\end{tabular}
\caption{Adopted choice of the computation of fluxes close to the boundary displayed in Fig.~\ref{fig:sketch_bnd}, depending on the order of the space discretization and the sign of the kinetic velocity $a_j$. $\Phi_{j, 1+1/2}^{2,*}$ is defined in \eqref{eq:flux_O2*}.}
\label{tab:bnds}
\end{table}

Finally note that with this finite-volume consideration, the extension of the example of Fig.~\ref{fig:sketch_bnd} to all the cases encountered in multi-dimensions on a Cartesian mesh is straightforward. There is no difficulty related to edges or corners because the only quantity that we impose is the flux across a cell.

\subsection{Equivalent Jin-Xin schemes}
\label{sec:eq_JinXin}

In Sec.~\ref{sec:regul_as_jinxin}, it has been shown that when a regularized collision matrix is considered, the kinetic system \eqref{eq:kinetic_with_Omega} is equivalent to a Jin-Xin system on the variables as $(\bbu^\varepsilon, \bbv_1^\varepsilon, \dots, \bbv_d^\varepsilon)$, thus reducing the rank of the system from $kp$ to $(d+1)p$. Le us now note the Jin-Xin variables $\mathbf{U}^\varepsilon = [\bbu^\varepsilon, \bbv_1^\varepsilon, \dots, \bbv_d^\varepsilon]^T$. The key relation allowing us to compute $\mathbf{U}^\varepsilon$ from $\bF$ and vice-versa when $\mathbb{H} \bF =  \mathbb{H}\M (\P \bF)=\mathbf{0}$ is
\begin{align}
	\mathbf{U}^\varepsilon = \overline{\momentsmatrix} \bF, \qquad
	\bF = \overline{\momentsmatrix}^+
	\mathbf{U}^\varepsilon,
\end{align}  
where $\overline{\momentsmatrix}$ is the reduced moments matrix and $\overline{\momentsmatrix}^+$ its pseudo inverse. In this section, we show how this can be used to reduce the memory cost of the numerical methods proposed above.

\subsubsection{First-order IMEX scheme}

We start with the first-order IMEX scheme of Sec.~\ref{sec:first_order_time}. Using a similar demonstration as in Sec.~\ref{sec:regul_as_jinxin}, we can show that Eqs.~\eqref{eq:first-order_IMEX_PF}-\eqref{eq:first_order_IMEX_F} are equivalent to the following system on $\mathbf{U}^\varepsilon$:
\begin{align}
	& \bbu^\varepsilon(t_{n+1}) = \bbu^\varepsilon(t_n) - \Delta t \sum_{i=1}^d \P \Lambda_i \delta_{x_i}^q \left( \overline{\momentsmatrix}^+ \mathbf{U}^\varepsilon(t_n) \right), \\
	& \begin{pmatrix}
		\bbv_1^\varepsilon \\ \vdots \\ \bbv_d^\varepsilon
	\end{pmatrix} (t_{n+1}) = 
	\left[ \mathbf{I}_{dp} + \hat{\mathbf{C}}_{n+1}^{-1} \right]^{-1} \left\{ \hat{\mathbf{C}}_{n+1}^{-1} \left[
	\begin{pmatrix}
		\bbv_1^\varepsilon \\ \vdots \\ \bbv_d^\varepsilon
	\end{pmatrix}(t_n) - \Delta t \sum_{i=1}^d
	\begin{pmatrix}
		\P \Lambda_1 \\ \vdots \\ \P \Lambda_d
	\end{pmatrix} \Lambda_i \delta_{x_i}^q \left( \overline{\momentsmatrix}^+ \mathbf{U}^\varepsilon(t_n) \right) \right] + \begin{pmatrix}
		\bbf_1^{n+1} \\ \vdots \\ \bbf_d^{n+1}
	\end{pmatrix} \right\},
\end{align}
where
\begin{align}
	\hat{\mathbf{C}}_{n+1} = \Delta t \frac{||\Lambda||}{\ell} \frac{\tilde{\mathbf{C}}}{\varepsilon}(\bbu^\varepsilon(t_{n+1})), \qquad \bbf_i^{n+1} = \bbf_i(\bbu^\varepsilon(t_{n+1})).
\end{align}
This is an explicit scheme for computing the variable $\mathbf{U}^\varepsilon(t_{n+1})$ from $\mathbf{U}^\varepsilon(t_{n})$. Note that it does not involve the Maxwellian $\M$, but the fluxes $\bbf_i$.

\subsubsection{High-order DeC schemes}
 Considering now the arbitrary high-order schemes of Sec.~\ref{sec:DeC}, we keep the same notation as before and write with a $\hat{\cdot}$ superscript any vector containing the information of all the subtimenodes, \textit{i.e.}
\begin{align}
	\hat{\bbu}^\varepsilon = (\mathbf{I}_s \otimes \P) \hat{\bF}, \qquad \hat{\bbv}_i^\varepsilon = ( \mathbf{I}_s \otimes \P \Lambda_i) \hat{\bF}, \qquad
	\hat{\mathbf{U}}^\varepsilon = (\mathbf{I}_s \otimes \overline{\momentsmatrix}) \hat{\bF}.
\end{align}
A similar work leads to the following algorithm on $\hat{\mathbf{U}}^\varepsilon$:
\begin{enumerate}
	\item Define $\hat{\mathbf{U}}^{\varepsilon, (0)} = [\mathbf{U}^\varepsilon(t_n), \dots, \mathbf{U}^\varepsilon(t_n) ]^T$ (repeated $s$ times),
	\item Repeat the following iterations for $0 \leq p \leq q-1$:
\begin{align}
	& \hat{\bbu}^{\varepsilon, (p+1)} = \hat{\bbu}^{\varepsilon, (0)} -(\mathbf{I}_s \otimes \P ) \hat{\mathbf{A}} \sum_{i=1}^d \hat{\Lambda}_i \delta_{x_i}^q \left( (\mathbf{I}_s \otimes \overline{\momentsmatrix}^+) \hat{\mathbf{U}}^{\varepsilon, (p)} \right), \\
	& \begin{pmatrix}
		\hat{\bbv}_1^\varepsilon \\ \vdots \\ \hat{\bbv}_d^\varepsilon
	\end{pmatrix}^{(p+1)} = \hat{\mathbf{C}}^{-1}_{(p+1)} \left[ \hat{\mathbf{C}}^{-1}_{(p+1)} + \tilde{\mathbf{A}} \right]^{-1} \left[ 
	\begin{pmatrix}
		\hat{\bbv}_1^\varepsilon \\ \vdots \\ \hat{\bbv}_d^\varepsilon
	\end{pmatrix}^{(0)} - \tilde{\mathbf{A}} 
	\sum_{i=1}^d \mathbf{I}_s \otimes 
	\begin{pmatrix}
		\P \Lambda_1 \\ \vdots \\ \P \Lambda_d
	\end{pmatrix}
	\hat{\Lambda}_i \delta_{x_i}^q \left( (\mathbf{I}_s \otimes \overline{\momentsmatrix}^+) \hat{\mathbf{U}}^{\varepsilon, (p)} \right) \right] \nonumber \\
	& \qquad \qquad \qquad + \tilde{\mathbf{A}} \left[ \hat{\mathbf{C}}_{(p+1)}^{-1} + \tilde{\mathbf{A}} \right]^{-1}
	\begin{pmatrix}
		\bbf_1(\hat{\bbu}^{\varepsilon, (p+1)}) \\ \vdots \\ \bbf_d(\hat{\bbu}^{\varepsilon, (p+1)})
	\end{pmatrix}, \label{eq:equivalent_jinxin_flux}
\end{align}
where
\begin{align}
	\hat{\mathbf{C}}_{(p)} = \frac{||\Lambda ||}{\ell} \begin{pmatrix}
		\tilde{\mathbf{C}}/\varepsilon(\bbu_1^{\varepsilon, (p)}) & \dots & 0 \\
		\vdots & \ddots & \vdots \\
		0 & \dots & \tilde{\mathbf{C}}/\varepsilon(\bbu_s^{\varepsilon, (p)})
	\end{pmatrix}, \qquad
	\tilde{\mathbf{A}} = \Delta t \mathbf{A} \otimes \mathbf{I}_{dp}.
\end{align}
\item When a RK scheme like Lobato IIIC is considered, the updated variables are obtained as $\mathbf{U}^\varepsilon(t_{n+1}) = \mathbf{U}_s^{\varepsilon, (q)}$.
\end{enumerate}

\begin{remark}
\label{rk:Jin-Xin_disc1}
	In the above numerical methods, the convective terms appear in the form $\overline{\momentsmatrix} \Lambda_i \delta_{x_i}^q (\overline{\momentsmatrix}^+ \mathbf{U}^\varepsilon)$, consistent with the convective terms of the equivalent Jin-Xin model \eqref{eq:regul_as_JinXin}. Note that the quantity $\delta_{x_i}^q (\overline{\momentsmatrix}^+ \mathbf{U}^\varepsilon)$ is different from $ \overline{\momentsmatrix}^+ \delta_{x_i}^q \mathbf{U}^\varepsilon$, because the space derivatives $\delta_{x_i}^q$ are only defined in the context of distributions since the upwinding depends on the sign of the kinetic wave (see Sec.~\ref{sec:space_disc}). It is therefore necessary to multiply the Jin-Xin variables by the pseudo-inverse matrix $\overline{\momentsmatrix}^+$ before computing the space derivatives. In this sense, the numerical methods of this section can be seen as algorithms to solve the Jin-Xin system \eqref{eq:regul_as_JinXin}, where the discussion about kinetic waves only enters into the approximation of space derivatives.
\end{remark}

\begin{remark}
	Following Remark~\ref{rk:Jin-Xin_disc1}, we could have adopted another numerical discretization of the space gradients appearing in \eqref{eq:regul_as_JinXin} by considering the eigenspace of the transport matrice $\mathbf{A}_i$ instead of the kinetic waves. This would have lead to a different class of numerical methods which do not share the entropy properties of the kinetic schemes developed in this work, because the transport matrices $\mathbf{A}_i$ are not simultaneously diagonalizable, meaning that the transport term in the eigenspace of $\mathbf{A}_i$ is not a mere advection. 
\end{remark}

\section{Numerical validations}
\label{sec:numerical_validations}

In this section, the proposed models are numerical assessed on some standard cases of the literature for transport-diffusion problem. The advection-diffusion of a scalar variable is first investigated, then the Navier-Stokes equations for fluid dynamics are considered. In any case, the four-velocity model in two dimensions of Example \ref{ex:moments_matrix_D2Q4} is adopted and a regularization is performed in the moments space related to the matrix $\momentsmatrix$. This leads to $\mathbb{H} \bF = \mathbb{H} \M(\P \bF) = \mathbf{0}$ with 
\begin{align}
	\mathbb{H} = [1, 1, -1, -1] \otimes \mathbf{I}_p.
\end{align}
For the sake of computational efficiency, the equivalent Jin-Xin schemes proposed in Sec~\ref{sec:eq_JinXin} have been implemented in a massively parallel Fortran code. The kinetic speed $a$ is always chosen so that the sub-characteristic condition is satisfied. Its value will be specified in each case, together with the Knudsen number $\varepsilon$ computed with \eqref{eq:def_kn}. In all cases, a uniform Cartesian mesh is considered with mesh size $\Delta x$. Four numerical methods are considered:
\begin{enumerate}
	\item A first-order scheme based on the IMEX method of Sec.~\ref{sec:first_order_time} with the space discretization $\delta_{x_i}^1$. In this case, we set the kinetic CFL number to $\lambda = a\Delta t/\Delta x = 1$.
	\item A second-order scheme based on the DeC method of Sec.~\ref{sec:DeC} involving the second-order Lobato IIIC scheme and the space discretizations $\delta_{x_i}^2$. To reach a second-order of accuracy, the DeC iterations are performed twice. In this case, we set the kinetic CFL number to $\lambda = a\Delta t/\Delta x = 0.8$.
	\item A fourth-order scheme based on the DeC method of Sec.~\ref{sec:DeC} involving the fourth-order Lobato IIIC scheme and the space discretizations $\delta_{x_i}^4$. To reach a fourth-order of accuracy, the DeC iterations are performed four times. In this case, we set the kinetic CFL number to $\lambda = a\Delta t/\Delta x = 1$.
\end{enumerate}

\subsection{Scalar advection-diffusion equation}

We first assess the proposed method for the resolution of scalar advection-diffusion equations in the form
\begin{align}
	\dpar{u}{t} + \dpar{f_1(u)}{x_1} + \dpar{f_2(u)}{x_2} = \alpha \left( \frac{\partial^2 u}{\partial x_1^2} + \frac{\partial^2 u}{\partial x_2^2} \right),
\end{align}
where $f_1(u) = c_1 u$, $f_2(u) = c_2 u$, $c_1$ and $c_2$ are constant advection velocities and $\alpha \geq 0$ is a diffusion parameter. In this case, the sub-characteristic condition reads
\begin{align}
	a > 2\max(c_1, c_2).
\end{align}

A $(L \times L)$ square periodic domain with $L=1$ is initialized with the following Gaussian state
\begin{align}
    u(x_1,x_2,0) = 1 + 0.01 \exp \left( - \frac{(x_1-0.5)^2 + (x_2-0.5)^2}{\delta^2} \right),
\end{align}
where $\delta = 0.1$ is the characteristic length of this problem. Following \eqref{eq:def_kn}, a Knudsen number is defined as
\begin{align}
	\varepsilon = \frac{\alpha}{a \delta},
\end{align}
and the advection velocities are set as $c_1 = c_2 = c =10$, in order to investigate the transport of the Gaussian in the diagonal direction. A square mesh of $(N \times N)$ equally spaced points is considered, with mesh size $\Delta x = 1/N$. The solution can be compared with the exact solution given by
\begin{align}
	   u_{exact}(x,t) = 1+0.01 \frac{1}{1+4\alpha t/\delta^2} \exp \left( -\frac{(x_1-0.5-c_1t)^2 + (x_2-0.5-c_2t)^2}{\delta^2 + 4\alpha t} \right),
\end{align}
allowing us to compute a $L^2$ error at time $t$ as
\begin{align}
    L^2 = \sqrt{\frac{\sum_i (u(\boldsymbol{x}_i, t) - u_{exact}(\boldsymbol{x}_i, t))^2}{\sum_i u_{exact}(\boldsymbol{x}_i, t)^2}},
\end{align}
where the sums are performed over all the points of the two-dimensional Cartesian mesh. Figure \ref{fig:Order_ADE} displays the $L^2$ errors obtained at time $t=0.005$ for meshes ranging from $N=10$ to $N=1280$ in logarithmic scales, for the three schemes under consideration. Three cases are considered:
\begin{enumerate}
	\item[(a)] $\alpha = 0.01$ and $a=2.1c = 21$, leading to a Knudsen number $\varepsilon \approx 4.8\times 10^{-3}$,
	\item[(b)] $\alpha = 0.01$ and $a=100c = 1000$, leading to a Knudsen number $\varepsilon = 10^{-4}$,
	\item[(c)] $\alpha = 0$ and $a=2.1c = 21$, leading to a Knudsen number $\varepsilon = 0$.
\end{enumerate} 

\begin{figure}[h!]
    \centering
    \begin{subfigure}{0.32\textwidth}
    \centering
    \includegraphics{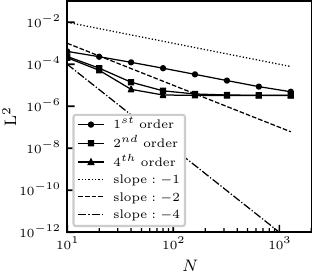}
    \caption{$\alpha=0.01$, $a=2.1c$ ($\varepsilon \approx 4.8\ 10^{-3}$)}
    \end{subfigure}
    \begin{subfigure}{0.32\textwidth}
    \centering
    \includegraphics{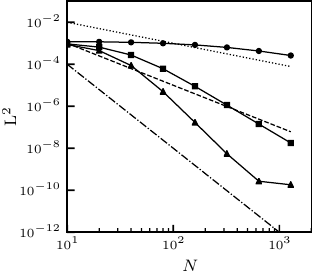}
    \caption{$\alpha=0.01$, $a=100c$ ($\varepsilon = 10^{-4}$)}
    \end{subfigure}
    \begin{subfigure}{0.32\textwidth}
    \centering
    \includegraphics{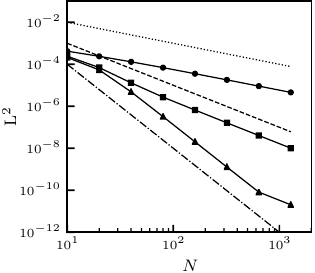}
    \caption{$\alpha=0$, $a=2.1c$ ($\varepsilon = 0$)}
    \end{subfigure}
    \caption{Mesh convergence study of the advection-diffusion of an initial Gaussian shape at time $t=0.005$.}
    \label{fig:Order_ADE}
\end{figure}

\begin{table}[!ht]
\begin{subtable}{1.\textwidth}
\centering
\begin{tabular}{|p{0.05\textwidth}||>{\centering}p{0.16\textwidth}>{\centering}p{0.04\textwidth}|>{\centering}p{0.16\textwidth}>{\centering}p{0.04\textwidth}|>{\centering}p{0.16\textwidth}>{\centering\arraybackslash}p{0.04\textwidth}|}
\hline
 & \multicolumn{2}{|c|}{First-order} & \multicolumn{2}{c|}{Second-order} & \multicolumn{2}{c|}{Fourth-order} \\
\hline
$h$ & $L^2$ & $r$ & $L^2$ & $r$ & $L^2$ & $r$ \\
\hline
$10$ & $4.10441781\ 10^{-4}$ & - & $2.35184625\ 10^{-4}$ & - & $2.05268145\ 10^{-4}$ & - \\
$20$ & $2.29191590\ 10^{-4}$ & $0.84$ & $6.54767191\ 10^{-5}$ & $1.84$ & $5.18934685\ 10^{-5}$ & $1.98$ \\
$40$ & $1.23974890\ 10^{-4}$ & $0.89$ & $1.40309659\ 10^{-5}$ & $2.22$ & $6.40803371\ 10^{-6}$ & $3.02$ \\
$80$ & $6.50497422\ 10^{-5}$ & $0.93$ & $5.45364755\ 10^{-6}$ & $1.36$ & $3.44739780\ 10^{-6}$ & $0.89$ \\
$160$ & $3.31818933\ 10^{-5}$ & $0.97$ & $3.85747480\ 10^{-6}$ & $0.50$ & $3.30981625\ 10^{-6}$ & $0.06$ \\
$320$ & $1.67915333\ 10^{-5}$ & $0.98$ & $3.44616432\ 10^{-6}$ & $0.16$ & $3.30030830\ 10^{-6}$ & $0.00$ \\
$640$ & $8.62269212\ 10^{-6}$ & $0.96$ & $3.33713282\ 10^{-6}$ & $0.05$ & $3.29967606\ 10^{-6}$ & $0.00$ \\
$1280$ & $4.91867644\ 10^{-6}$ & $0.81$ & $3.30878501\ 10^{-6}$ & $0.01$ & $3.29963553\ 10^{-6}$ & $0.00$ \\
\hline 
\end{tabular}
\caption{$\alpha=0.01$, $a=2.1c$ ($\varepsilon \approx 4.8\ 10^{-3}$)}
\end{subtable} \\ \vspace{1mm} \\
\begin{subtable}{1.\textwidth}
\centering
\begin{tabular}{|p{0.05\textwidth}||>{\centering}p{0.16\textwidth}>{\centering}p{0.04\textwidth}|>{\centering}p{0.16\textwidth}>{\centering}p{0.04\textwidth}|>{\centering}p{0.16\textwidth}>{\centering\arraybackslash}p{0.04\textwidth}|}
\hline
 & \multicolumn{2}{|c|}{First-order} & \multicolumn{2}{c|}{Second-order} & \multicolumn{2}{c|}{Fourth-order} \\
\hline
$h$ & $L^2$ & $r$ & $L^2$ & $r$ & $L^2$ & $r$ \\
\hline
$10$ & $1.18168727\ 10^{-3}$ & - & $9.36363462\ 10^{-4}$ & - & $8.73349525\ 10^{-4}$ & - \\ 
$20$ & $1.17685430\ 10^{-3}$ & $0.01$ & $6.56543672\ 10^{-4}$ & $0.51$ & $4.48079316\ 10^{-4}$ & $0.96$ \\ 
$40$ & $1.11082364\ 10^{-3}$ & $0.08$ & $2.76720000\ 10^{-4}$ & $1.25$ & $8.73161790\ 10^{-5}$ & $2.36$ \\ 
$80$ & $9.99372824\ 10^{-4}$ & $0.15$ & $6.13885711\ 10^{-5}$ & $2.17$ & $4.97815056\ 10^{-6}$ & $4.13$ \\ 
$160$ & $8.37946655\ 10^{-4}$ & $0.25$ & $8.89695848\ 10^{-6}$ & $2.79$ & $1.69257770\ 10^{-7}$ & $4.88$ \\ 
$320$ & $6.35500452\ 10^{-4}$ & $0.40$ & $1.13974353\ 10^{-6}$ & $2.96$ & $5.47995040\ 10^{-9}$ & $4.95$ \\ 
$640$ & $4.29897683\ 10^{-4}$ & $0.56$ & $1.43016674\ 10^{-7}$ & $2.99$ & $2.70040658 \ 10^{-10}$ & $4.34$ \\ 
$1280$ & $2.61670647\ 10^{-4}$ & $0.72$ & $1.79102733\ 10^{-8}$ & $3.00$ & $1.82973511\ 10^{-10}$ & $0.56$ \\ \hline
\end{tabular}
\caption{$\alpha=0.01$, $a=100c$ ($\varepsilon =10^{-4}$)}
\end{subtable} \\ \vspace{1mm} \\
\begin{subtable}{1.\textwidth}
\centering
\begin{tabular}{|p{0.05\textwidth}||>{\centering}p{0.16\textwidth}>{\centering}p{0.04\textwidth}|>{\centering}p{0.16\textwidth}>{\centering}p{0.04\textwidth}|>{\centering}p{0.16\textwidth}>{\centering\arraybackslash}p{0.04\textwidth}|}
\hline
 & \multicolumn{2}{|c|}{First-order} & \multicolumn{2}{c|}{Second-order} & \multicolumn{2}{c|}{Fourth-order} \\
\hline
$h$ & $L^2$ & $r$ & $L^2$ & $r$ & $L^2$ & $r$ \\
\hline
$10$ & $4.21656606\ 10^{-4}$ & - & $2.42351339\ 10^{-4}$ & - & $2.10426945\ 10^{-4}$ & - \\ 
$20$ & $2.39939665\ 10^{-4}$ & $0.81$ & $6.97979830\ 10^{-5}$ & $1.80$ & $5.40276211\ 10^{-5}$ & $1.96$ \\ 
$40$ & $1.31004777\ 10^{-4}$ & $0.87$ & $1.30520721\ 10^{-5}$ & $2.42$ & $4.86173362\ 10^{-6}$ & $3.47$ \\ 
$80$ & $6.92500030\ 10^{-5}$ & $0.92$ & $2.72682525\ 10^{-6}$ & $2.26$ & $3.21964570\ 10^{-7}$ & $3.92$ \\ 
$160$ & $3.55788052\ 10^{-5}$ & $0.96$ & $6.61919579\ 10^{-7}$ & $2.04$ & $2.03174351\ 10^{-8}$ & $3.99$ \\ 
$320$ & $1.81069823\ 10^{-5}$ & $0.97$ & $1.62950885\ 10^{-7}$ & $2.02$ & $1.27208179\ 10^{-9}$ & $4.00$ \\ 
$640$ & $9.12595098\ 10^{-6}$ & $0.99$ & $4.05774304\ 10^{-8}$ & $2.01$ & $8.17174630\ 10^{-11}$ & $3.96$ \\ 
$1280$ & $4.58470481\ 10^{-6}$ & $0.99$ & $1.01342336\ 10^{-8}$ & $2.00$ & $2.03920378\ 10^{-11}$ & $2.00$ \\ \hline
\end{tabular}
\caption{$\alpha=0$, $a=2.1c$ ($\varepsilon =0$)}
\end{subtable}
\caption{Orders of convergence for the advection-diffusion problem and two-wave model for orders 1, 2
and 4. The final time is $T=0.005$. The wave velocity $a$ is varied to exhibit its effect on the $\mathcal{O}(\varepsilon^2)$ consistency error, which appears as a plateau in the $L^2$ error of the high-order schemes.}
\label{tab:Order_ADE}
\end{table}

Note that in case (c), the numerical scheme of Sec.~\ref{sec:eq_JinXin} can be considerably simplified. Indeed, $\alpha =0$ leads to $\hat{\mathbf{C}}_{(p)}^{-1} = \mathbf{0}_{dkp \times dkp}$ so that \eqref{eq:equivalent_jinxin_flux} simply becomes
\begin{align}
	\begin{pmatrix}
		\bbv_1^\varepsilon \\ \bbv_2^\varepsilon
	\end{pmatrix}^{(p+1)}
	 =
	 \begin{pmatrix}
	 	\bbf_1(\hat{\bbu}^{\varepsilon, (p+1)}) \\ \bbf_2(\hat{\bbu}^{\varepsilon, (p+1)})
	 \end{pmatrix}.
\end{align}
The first-, second- and fourth-order schemes are then similar to what was proposed in \cite{AbgrallK}, in which asymptotic preservation was demonstrated. 


The following conclusions can be drawn from Fig.~\ref{fig:Order_ADE}:
\begin{itemize}
	\item In cases (a) and (b), a plateau is observed for the finest meshes, in agreement with the $\mathcal{O}(\varepsilon^2)$ consistency error of the present kinetic model with the target advection-diffusion equation. We observe that the value of this plateau decreases as the Knudsen decreases.
	\item In case (c), no plateau is observed because the numerical scheme with $\alpha=0$ is consistent with the target hyperbolic system, as shown in~\cite{Torlo}.
	\item As shown in \cite{wissocq2023Kinetic}, the numerical accuracy of the first-order scheme decreases as the kinetic speed increases.
	\item As shown in \cite{wissocq2023Kinetic}, case (b) exhibits superconvergence for the second-order scheme with an apparent third-order of convergence. Same observation is drawn for the fourth-order scheme, for which an apparent fifth-order of convergence is measured.
	\item With case (c), the measured orders of convergence are in agreement with the expected orders of accuracy of the schemes.
\end{itemize}

To better support these observations, the measured $L^2$ errors and slopes $r$ are provided in Table~\ref{tab:Order_ADE}. \newline

An important particularity of the proposed kinetic schemes is the existence of a non-zero but controllable consistency error in $\mathcal{O}(\varepsilon^2)$. As discussed in \cite{wissocq2023Kinetic}, this error is due to the fact that we do not want to target the transport-diffusion system in the strict limit $\varepsilon \rightarrow 0$ of a kinetic system, but as the first-order term of a Chapman-Enskog expansion for low, but non-zero, values of $\varepsilon$. In order to numerically validate the behavior of this consistency error, an asymptotic study is performed for increasing values of the kinetic velocities, from $a=2.1c$ to $a=67.2c$. For this study, the fourth-order scheme is considered with the finer mesh $N=1280$, in order to get rid of the numerical errors. $L^2$ errors and computed slopes $r$ are compiled in Table~\ref{tab:Order_Kn_ADE}. The expected $\mathcal{O}(\varepsilon^2)$ consistency error is clearly exhibited.

\begin{table}[!ht]
\centering
\begin{tabular}{|p{0.03\textwidth}|>{\centering}p{0.13\textwidth}>{\centering}p{0.13\textwidth}>{\centering}p{0.13\textwidth}>{\centering}p{0.13\textwidth}>{\centering}p{0.13\textwidth}>{\centering\arraybackslash}p{0.13\textwidth}|}
\hline
$a$ & $2.1c$ & $4.2c$ & $8.4c$ & $16.8c$ & $33.6c$ & $67.2c$ \\
\hline
$\varepsilon$ & $4.8\ 10^{-3}$ & $2.4\ 10^{-3}$ & $1.2\ 10^{-3}$ & $6.0\ 10^{-4}$ & $3.0\ 10^{-4}$ & $1.5\ 10^{-4}$ \\
\hline 
$L^2$ & $3.299636\ 10^{-6}$ & $1.263713\ 10^{-7}$ & $2.672607\ 10^{-8}$ & $6.445998\ 10^{-9}$ & $1.599668\ 10^{-9}$ & $4.012814\ 10^{-10}$ \\
$r$ & - & $4.71$ & $2.24$ & $2.05$ & $2.01$ & $2.00$ \\
\hline
\end{tabular}
\caption{Asymptotic study of the consistency error in Knudsen number $\varepsilon$ of the advection-diffusion of a Gaussian. In order to get rid of numerical errors, a fine mesh of $(1280 \times 1280)$ points is considered and simulations are performed with the fourth-order scheme. Knudsen number is defined as $\varepsilon = \alpha/(a \delta)$.}
\label{tab:Order_Kn_ADE}
\end{table}


\subsection{Navier-Stokes equations: Thermal Couette flow}
\label{sec:validation_Couette}

We now assess the ability of the proposed kinetic methods to approximate the two-dimensional Navier-Stokes equations for fluid dynamics with a hyperbolic condition on the time step. For this case, the fluxes of conserved variables $\bbf_1$ and $\bbf_2$ and the diffusion matrix $\mathbf{D}$ are provided in Example~\ref{ex:NS_2D}. The sub-characteristic condition reads
\begin{align}
	a > 2\, \mathrm{max}(\rho(\bbf_1'(\bbu)), \rho(\bbf_2'(\bbu))) = 2\, \mathrm{max}(u + c, v+c),
\end{align}
where $c=\sqrt{\gamma P/\rho}$ is the sound speed, $P$ is the thermodynamic pressure given by the ideal gas equation of state,
\begin{align}
	P = (\gamma-1)\left( E - \frac{\rho}{2}(u^2 + v^2) \right),
\end{align}
and $\gamma$ is the heat capacity ratio. In the present section, the kinetic schemes are evaluated on two thermal Couette flow test cases. In the first case, a one-dimensional domain of size $L=1$ is initialized with a flow at rest as follows:
\begin{align}
	\rho(\boldsymbol{x}, t=0) = 1, \qquad u(\boldsymbol{x}, t=0) = 0, \qquad v(\boldsymbol{x}, t=0) = 0, \qquad P(\boldsymbol{x}, t=0) = 1.
\end{align}
Two isothermal walls are considered as boundary conditions. On the left and right walls, the following variables are respectively imposed:
\begin{align}
	[u, v, T]_L = [0, 1.3\sqrt{\gamma}, 1] , \qquad [u, v, T]_R = [0, 0, 1], 
\end{align}
where $T=P/\rho$ is the fluid temperature. Moreover, the wall pressure is imposed by ensuring a zero pressure gradient in the direction normal to the wall, using a second-order approximation of the gradient. For the left wall, it reads
\begin{align}
	P_L = \frac{9}{8} P_1 - \frac{1}{8} P_2,
\end{align}
where $P_1$ is the pressure a the first cell close to the wall, and $P_2$ is the pressure of the second cell (see Fig.~\ref{fig:sketch_bnd}). Once the quantities $\bbu_b$ are known at both left and right boundaries, the procedure described in Sec.~\ref{sec:boundaries} can be applied to ensure consistent and mass conserving boundary conditions.

The other physical parameters of this test case are: $\gamma = 1.4$, $\mathrm{Pr} = 0.73$, $\mu = 0.01$ and $\lambda = -(2/3)\mu$. In order to enforce the sub-characteristic condition, the kinetic speed is updated at each iteration as $a = 2.1 \, \mathrm{max}(u+c, v+c)$, which dynamically adapts the time step. After time convergence is reached, the temperature profiles obtained with the first-, second- and fourth order schemes are displayed in Fig.~\ref{fig:Couette_isoT}. They are compared with the exact profile for isothermal walls,
\begin{align}
	T_{exact} = 1 + \frac{(\gamma-1) \mathrm{Pr}}{2 \gamma} v_L^2 \, \frac{x_1}{L} \left( 1-\frac{x_1}{L} \right).
\end{align}
In Fig~\ref{fig:Couette_isoT}-(a), the first-order scheme is evaluated with $N=100$, $N=200$ and $N=1000$ points. Even though a convergence towards the exact profile can be observed, the finest mesh is not sufficient to conclude on a correct approximation of the solution. On the other hand, the second- and fourth-order schemes are assessed with only $N=8$ points on Fig.~\ref{fig:Couette_isoT}-(b). This time, a very good agreement is observed with the exact solution, even with the second-order scheme with such a coarse mesh. The very good behavior of the second-order scheme can be explained by the fact that the analytical solution is a quadratic function of $x_1$, which can be correctly described by a second-order scheme. Also note that, in this test case, the way of prescribing the diffusive flux at the boundaries is of paramount importance since it drives the motion of the fluid. Looking at the fluid temperature at the first right and left points allows us to validate the diffusive fluxes imposed with the boundary conditions proposed in Sec.~\ref{sec:boundaries}. An important feature of these diffusive fluxes is that they are in agreement with a hyperbolic scaling of the time step.

\begin{figure}[h!]
    \centering
    \begin{subfigure}{0.48\textwidth}
    \centering
    \includegraphics{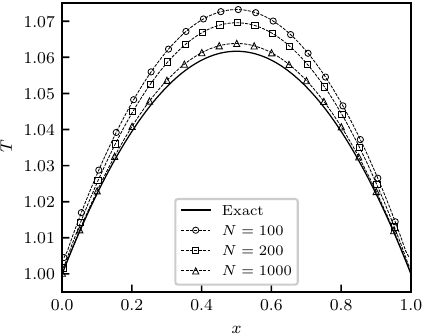}
    \caption{First-order scheme with $N=100,200,1000$ points.}
    \end{subfigure}
    \begin{subfigure}{0.48\textwidth}
    \centering
    \includegraphics{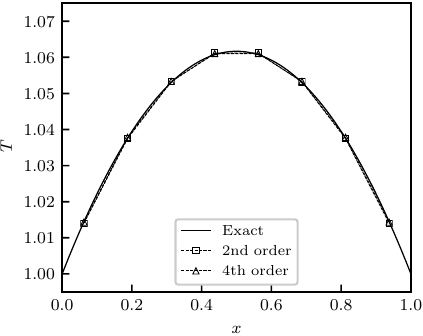}
    \caption{Second- and fourth-order schemes with $N=8$ points.}
    \end{subfigure}
    \caption{Temperature profiles of the Couette flow with isothermal-isothermal boundary conditions.}
    \label{fig:Couette_isoT}
\end{figure}

In the second Couette flow test case, the isothermal left wall is replaced by an adiabatic wall. This is done by replacing $T_L$ by
\begin{align}
	T_L = \frac{9}{8} T_1 - \frac{1}{8} T_2,
\end{align}
where $T_1$ is the fluid temperature at the first point close to the left wall and $T_2$ is the temperature at the second point. All the other parameters for this test case are kept unchanged. The exact solution for an adiabatic left wall reads
\begin{align}
	T_{exact} = 1 + \frac{(\gamma-1) \mathrm{Pr}}{2 \gamma} v_L^2 \, \left[ 1-\left( \frac{x_1}{L} \right)^2 \right].
\end{align}
Fig.~\ref{fig:Couette_adiab}-(a) displays the temperature profile obtained with the first-order scheme once time convergence is achieved. This is done for three meshes: $N=100$, $N=200$ and $N=1000$ points in the direction $x_1$. A similar conclusion as in Fig.~\ref{fig:Couette_isoT}-(a) can be drawn: a mesh convergence towards the exact solution can be observed, but $N=1000$ mesh points do not seem to be sufficient to accurately match the expected profile. Fig.~\ref{fig:Couette_adiab}-(b) shows that, with only $N=16$ points, the exact profile can be accurately recovered with the second- and fourth order schemes.

\begin{figure}[h!]
    \centering
    \begin{subfigure}{0.48\textwidth}
    \centering
    \includegraphics{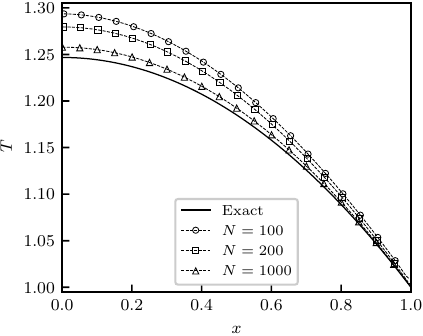}
    \caption{First-order scheme with $N=100,200,1000$ points.}
    \end{subfigure}
    \begin{subfigure}{0.48\textwidth}
    \centering
    \includegraphics{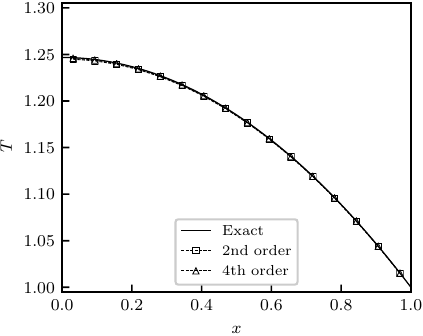}
    \caption{Second- and fourth-order schemes with $N=16$ points.}
    \end{subfigure}
    \caption{Temperature profiles of the Couette flow with adiabatic-isothermal boundary conditions.}
    \label{fig:Couette_adiab}
\end{figure}

A last important remark is that, even though this is not explicitly shown here, mass conservation has been numerically validated in the two thermal Couette flows of this section.

\subsection{Navier-Stokes equations: shock-boundary layer interaction}

We finally evaluate the ability of the proposed kinetic methods to accurately reproduce the complex patterns involved in the interactions between a shock wave and a boundary layer. This flow physics is known to be strongly affected by the Reynolds number. Therefore, it can be used to validate the introduction of diffusive terms (viscosity and thermal diffusion) in our model. The numerical setup of this case follows what was previously done in~\cite{Daru2009}. A two dimensional domain of size $(L_x \times L_y) = (1 \times 1/2)$ is considered with adiabatic static walls on left, bottom and right boundary conditions, and with a symmetry condition on the top boundary. In order to impose conserved variables on the adiabatic walls, the same strategy as in Sec.~\ref{sec:validation_Couette} is adopted and the distribution fluxes are imposed following Sec.~\ref{sec:boundaries}. Regarding the symmetry boundary conditions, it can be simply imposed by defining ghost nodes on the top of the computational domain. The values of the Jin-Xin variables $(\bbu^\varepsilon, \bbv_1^\varepsilon, \bbv_2^\varepsilon)$ are imposed at each iteration such that a symmetry is ensured at the top boundary condition. In the computational domain, the fluid is initialized as
\begin{align}
	[\rho, u, v, P] = 
	\begin{cases}
		[120, 0, 0, 120/\gamma] & \mathrm{if}\ x_1 \leq 0.5, \\
		[1.2, 0, 0, 1.2/\gamma] & \mathrm{else},
	\end{cases}
\end{align} 
in order to generate a left-moving rarefaction wave and right-moving contact and shock waves. Notably, the Mach number of the shock wave is equal to $2.37$. The physical parameters of the fluid are $\gamma = 1.2$, $\mathrm{Pr} = 0.73$ and the second viscosity (defined in Example~\ref{ex:NS_2D}) is set to $\lambda = -(2/3)\mu$. The dynamic viscosity is defined as $\mu=1/\mathrm{Re}$ where $\mathrm{Re}$ is the Reynolds number of the flow. Following the study of \cite{Daru2009}, four values are assessed for the Reynolds number: $\mathrm{Re}=200$, $\mathrm{Re} = 500$, $\mathrm{Re}=750$ and $\mathrm{Re=1000}$.

In this section, we only consider the fourth-order method based on the fourth-order Lobato IIIC scheme used in the DeC time integration of Sec.~\ref{sec:DeC}, with the fourth-order decentered gradients $\Phi^4$. The kinetic speed is set to a value as low as possible, \textit{i.e.}
\begin{align}
	a = 2.1 \, \mathrm{max}(u +c, v+c),
\end{align}
which leads to the following CFL number based on the acoustic speed: $\mathrm{CFL} = 0.48$. For the most viscous case ($\mathrm{Re}=200$), we consider a mesh with $(2000 \times 1000)$ points, while $(4000 \times 2000)$ points are considered in the other cases. The results are compared with the reference simulations of Daru \& Tenaud~\cite{Daru2009} performed with a one-step monotonicity preserving scheme of seventh-order of accuracy in the scalar case (OSMP7), with meshes ranging from $(1000 \times 500)$ to $(4000 \times 2000)$ points.

Fig.~\ref{fig:ShockBoundary_t06} displays the temperature contours obtained at time $t=0.6$ in the bottom-right corner of the simulation domain for the four Reynolds numbers considered, in comparison with the results obtained in \cite{Daru2009}. As expected, the reflection of the shock wave to the right boundary leads to a strong interaction with the bottom boundary layer, giving birth to the formation of a lambda-shape like shock pattern. A slip line generated from the triple point rolls up toward the right end corner, generating a ``bubble'' close to the bottom boundary. The shape of the interacting shock waves, the bubble and the cold jet moving to the left inside the boundary layer are in very good agreement with the reference simulation of \cite{Daru2009}. Notably, one can see that the larger the Reynolds number, the thinner the boundary layer.

\begin{figure}[h!]
   \begin{subfigure}{\textwidth}	
   \begin{subfigure}{0.49\textwidth}
    \includegraphics[width=0.95\textwidth, valign=t]{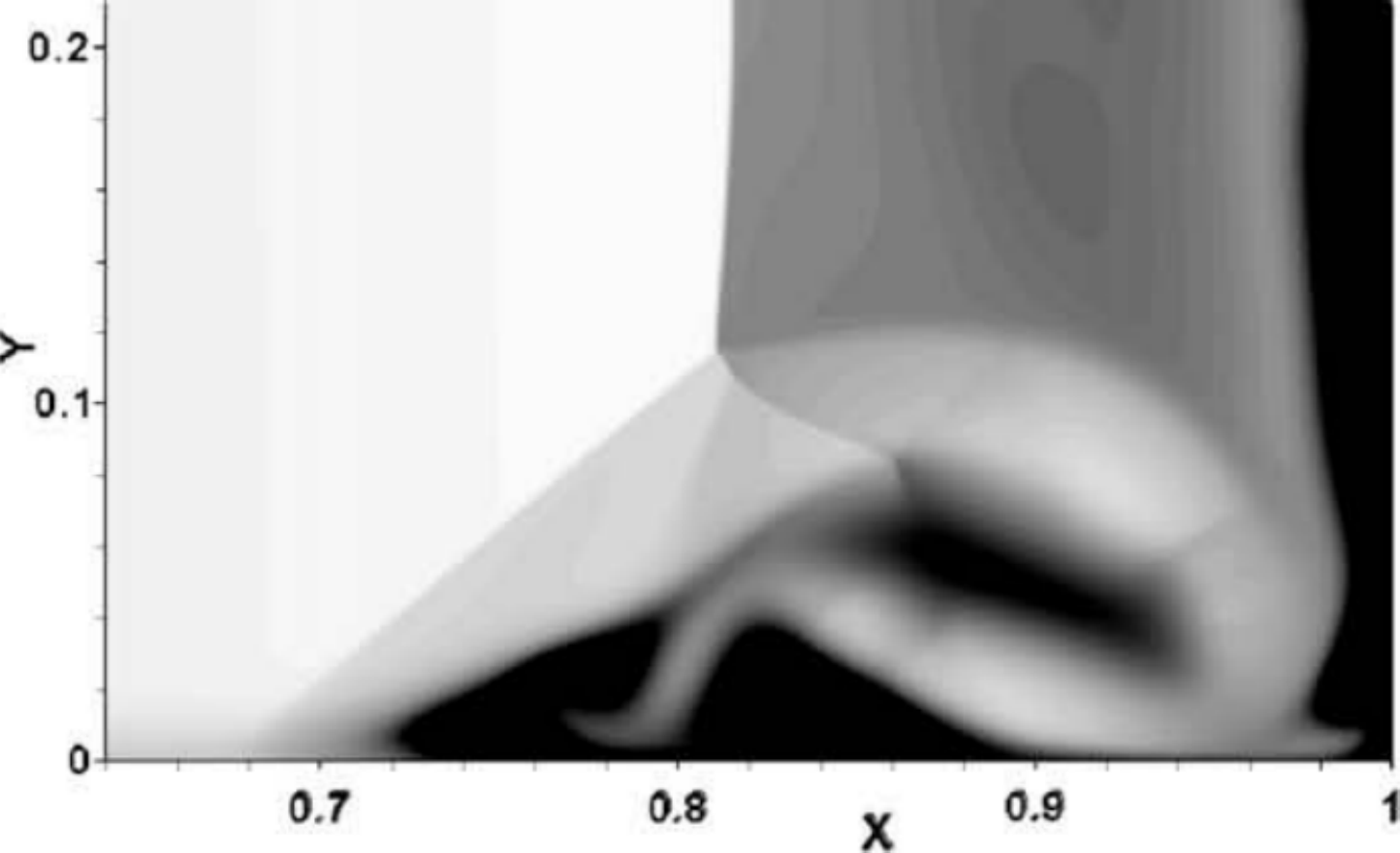}
    \end{subfigure}
    \begin{subfigure}{0.49\textwidth}
    \vspace{2.1mm}
    \includegraphics[width=0.98\textwidth,valign=t]{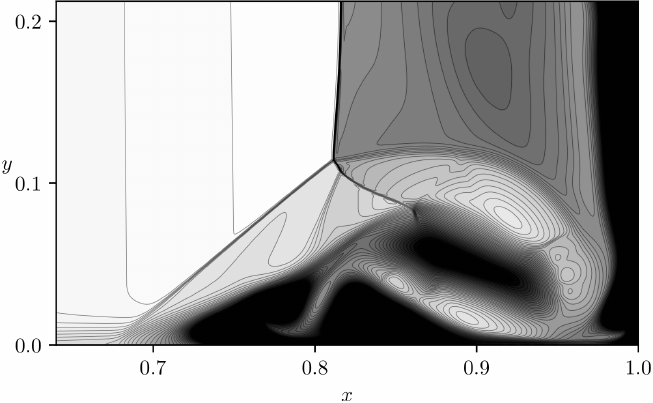}
    \end{subfigure}
   \end{subfigure}
   \begin{subfigure}{\textwidth}	
   \begin{subfigure}{0.49\textwidth}
    \includegraphics[width=0.95\textwidth, valign=t]{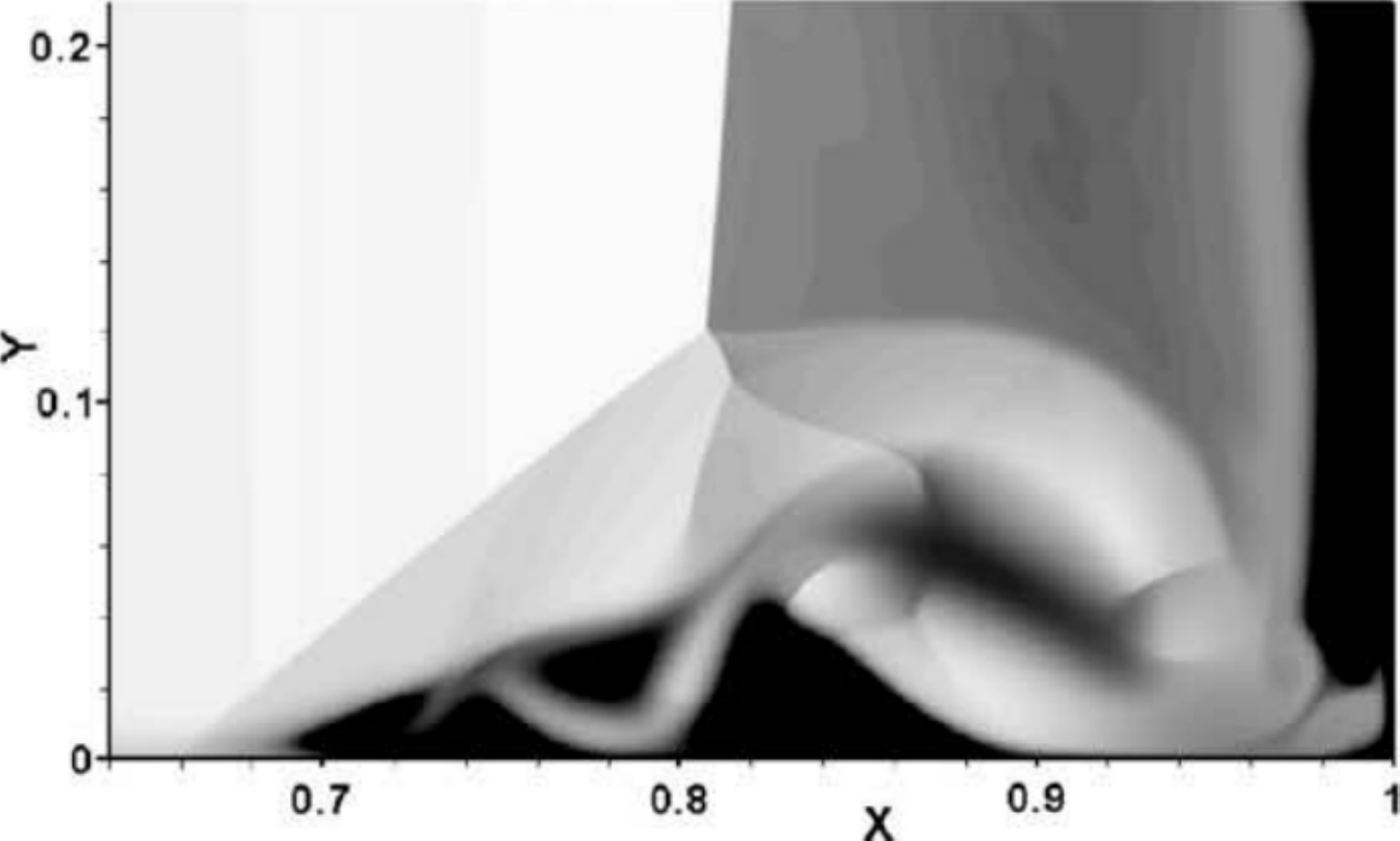}
    \end{subfigure}
    \begin{subfigure}{0.49\textwidth}
    \vspace{2.3mm}
    \includegraphics[width=0.98\textwidth,valign=t]{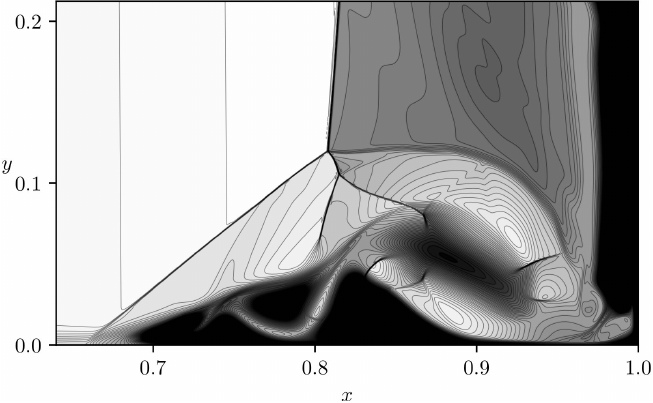}
    \end{subfigure}
   \end{subfigure}
   \begin{subfigure}{\textwidth}	
   \begin{subfigure}{0.49\textwidth}
    \includegraphics[width=0.95\textwidth, valign=t]{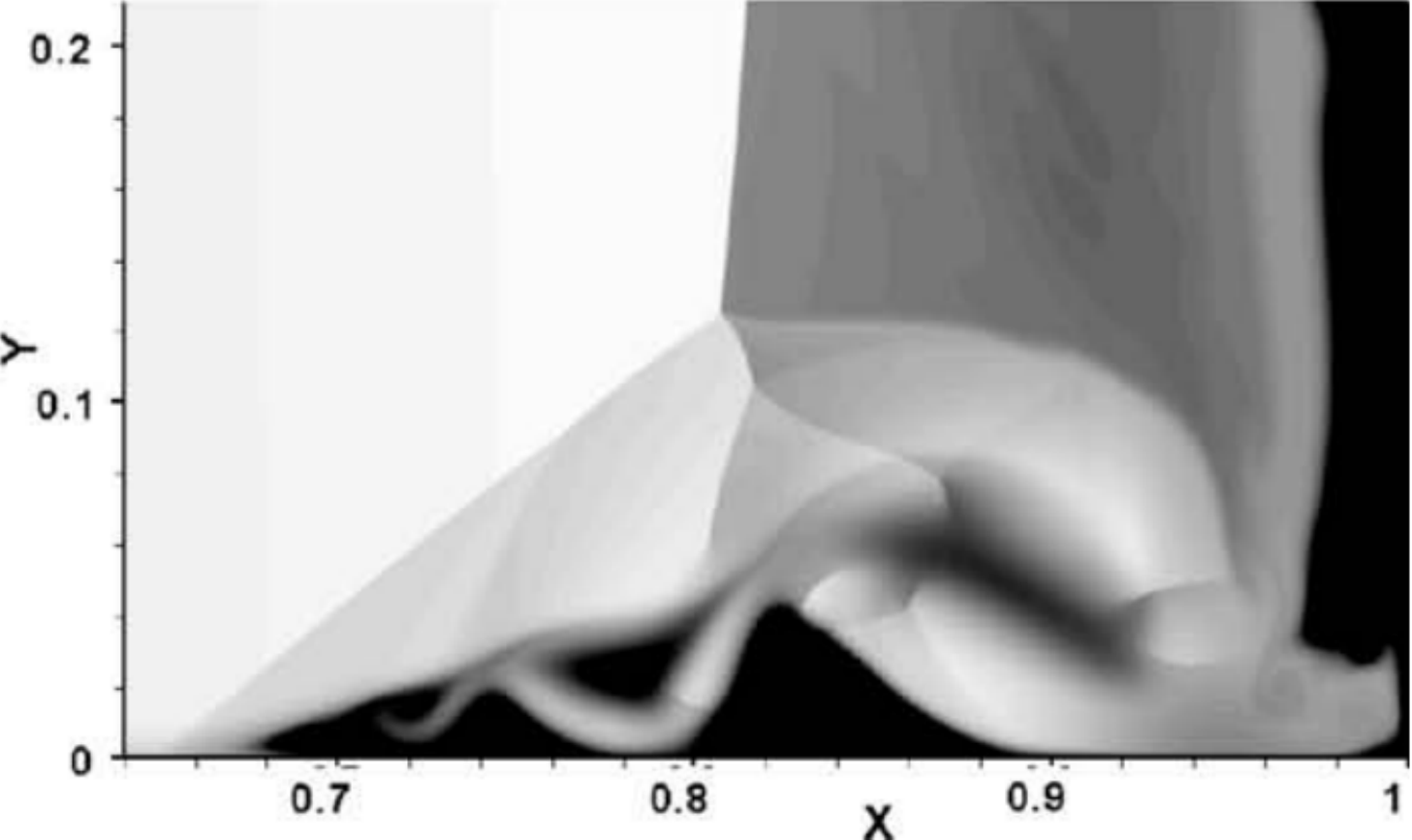}
    \end{subfigure}
    \begin{subfigure}{0.49\textwidth}
    \vspace{2.2mm}
    \includegraphics[width=0.98\textwidth,valign=t]{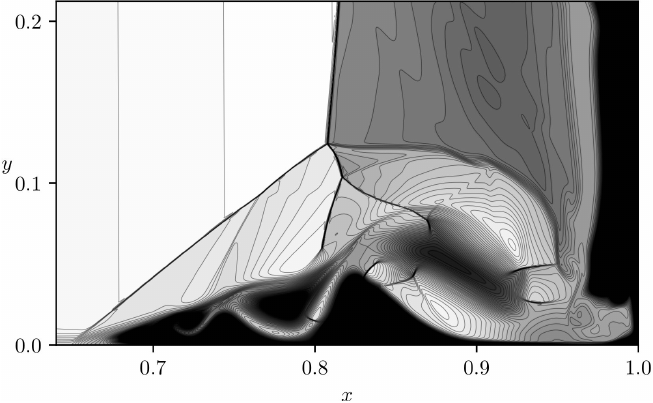}
    \end{subfigure}
   \end{subfigure}
   \begin{subfigure}{\textwidth}	
   \begin{subfigure}{0.49\textwidth}
    \includegraphics[width=0.95\textwidth, valign=t]{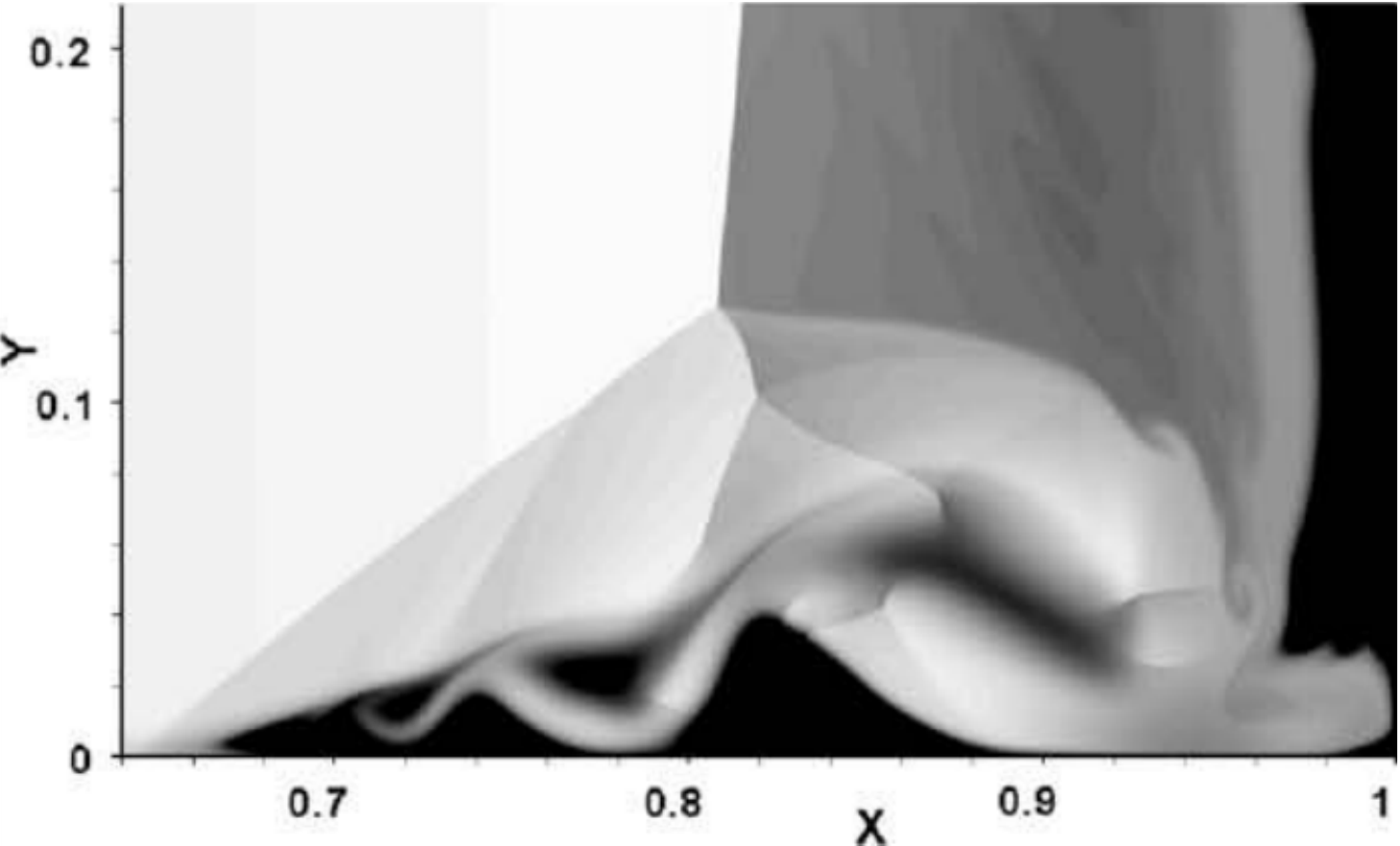}
    \end{subfigure}
    \begin{subfigure}{0.49\textwidth}
    \vspace{2.1mm}
    \includegraphics[width=0.98\textwidth,valign=t]{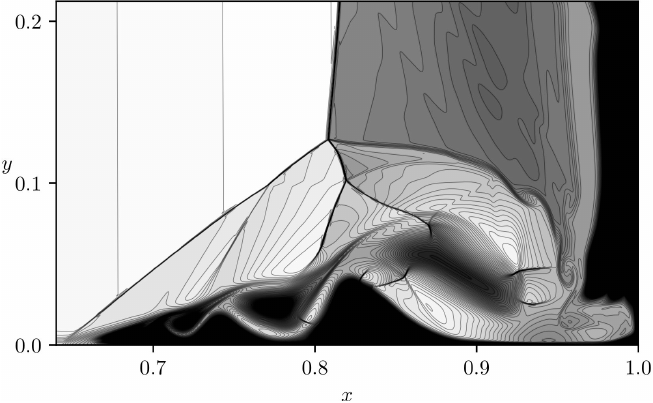}
    \end{subfigure}
   \end{subfigure}
   \caption{Contours of temperature (41 contour levels between 0.4 and 1.2) obtained at $t=0.6$ for several Reynolds numbers. Left: reference from~\cite{Daru2009} (OSMP7 scheme), right: present fourth-order scheme. From top to bottom: $\mathrm{Re}=200$, $\mathrm{Re}=500$, $\mathrm{Re}=750$ and $\mathrm{Re}=1000$.}
   \label{fig:ShockBoundary_t06}
\end{figure}

In Fig.~\ref{fig:ShockBoundary_t1}, similar temperature contours are displayed at time $t=1$. The overall pattern of the flow is, in any case, very similar to what was obtained with the OSMP7 scheme, which allows us to qualitatively validate the ability of the proposed kinetic scheme to deal with complex compressible and viscid flows.

\begin{figure}[h!]
	\begin{subfigure}{\linewidth}	
  \begin{subfigure}{0.49\textwidth}    \includegraphics[width=0.965\textwidth, valign=t]{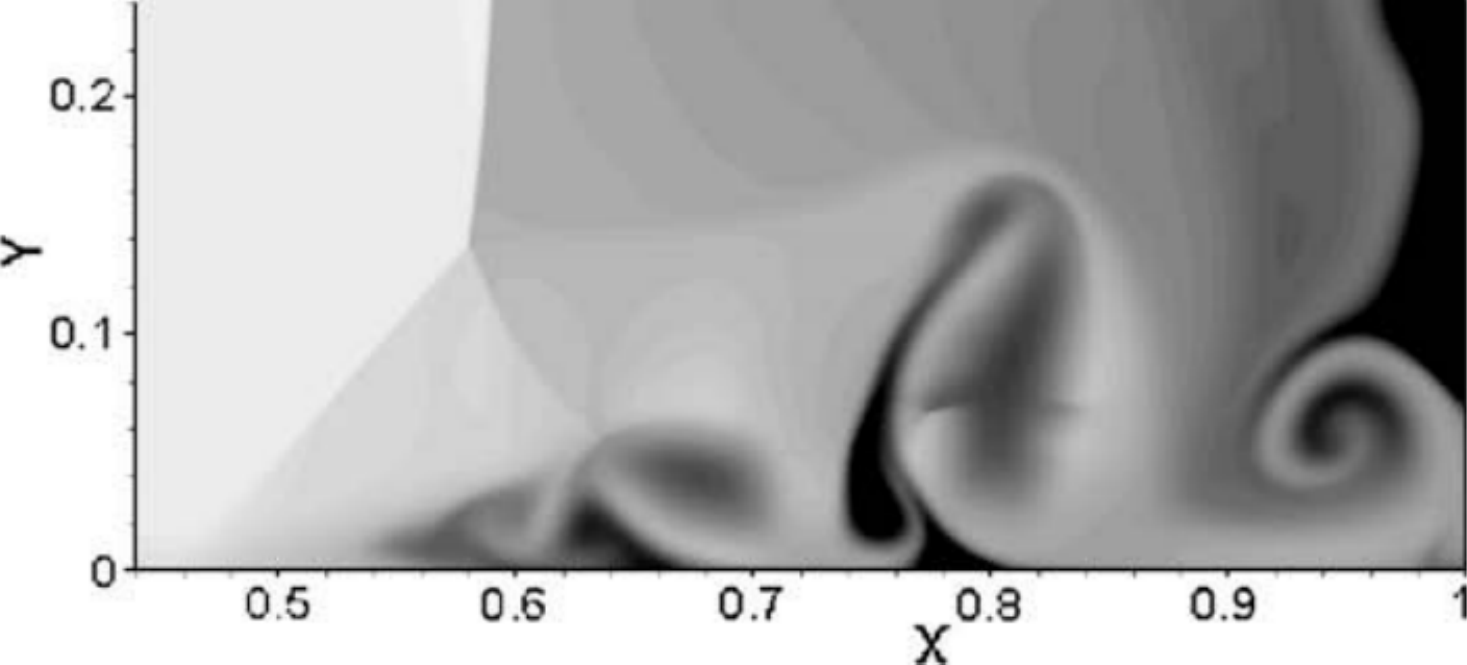}
    \end{subfigure}
    \begin{subfigure}{0.49\textwidth}
    \vspace{2mm}
    \includegraphics[width=0.98\textwidth, valign=t]{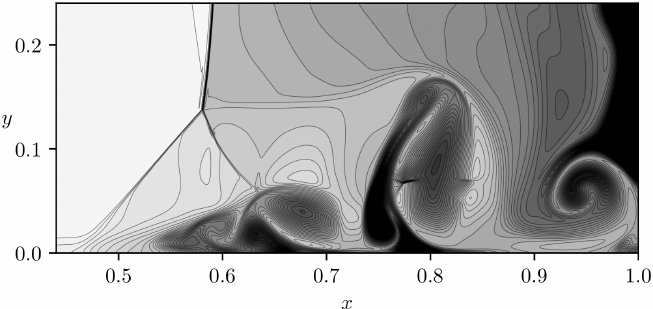}
    \end{subfigure}
   \end{subfigure}
	\begin{subfigure}{\linewidth}	
   \begin{subfigure}{0.49\textwidth}    \includegraphics[width=0.965\textwidth, valign=t]{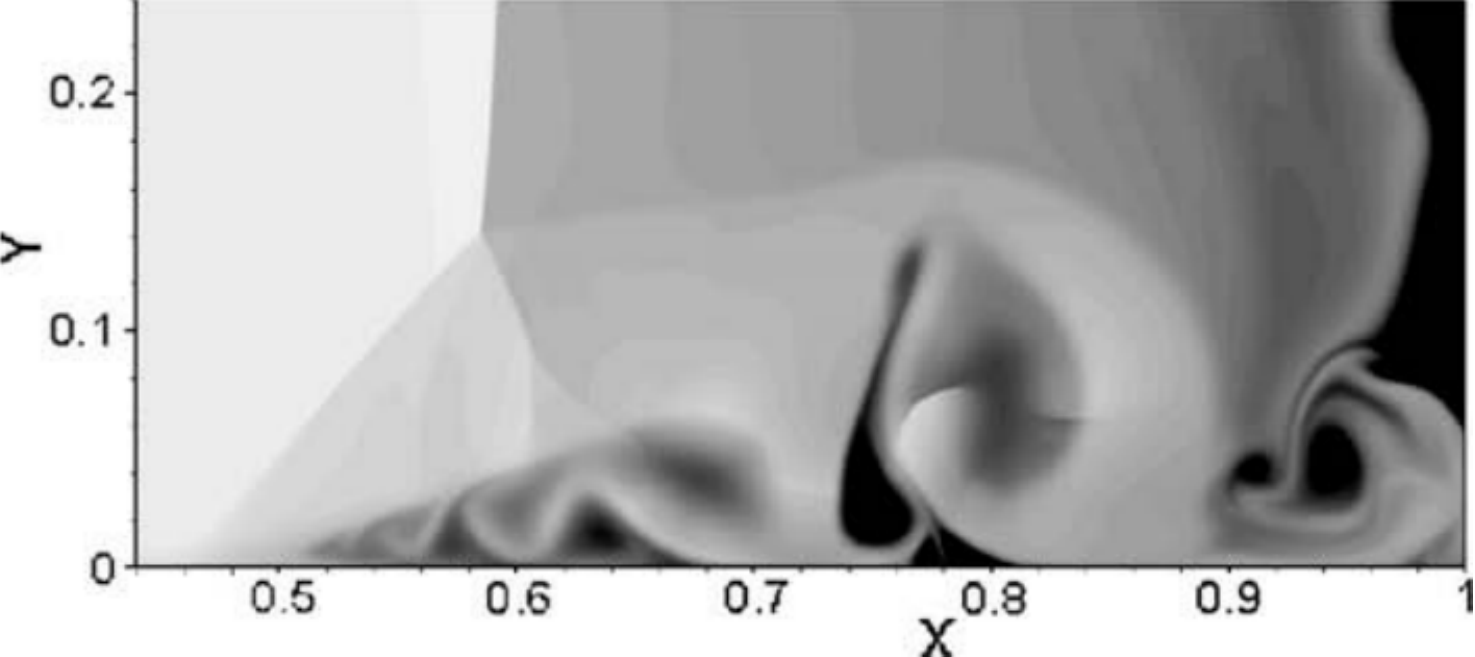}
    \end{subfigure}
    \begin{subfigure}{0.49\textwidth}
    \vspace{2mm}
    \includegraphics[width=0.98\textwidth, valign=t]{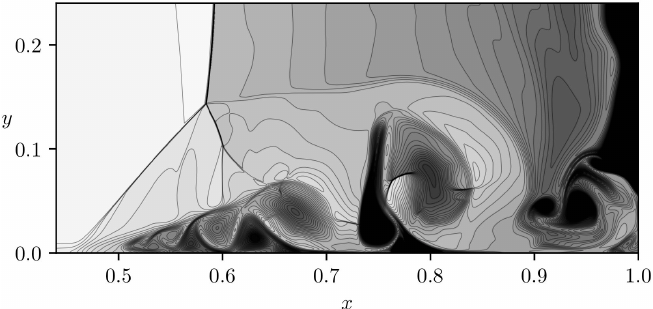}
    \end{subfigure}
   \end{subfigure}
	\begin{subfigure}{\linewidth}	
   \begin{subfigure}{0.49\textwidth}    \includegraphics[width=0.965\textwidth, valign=t]{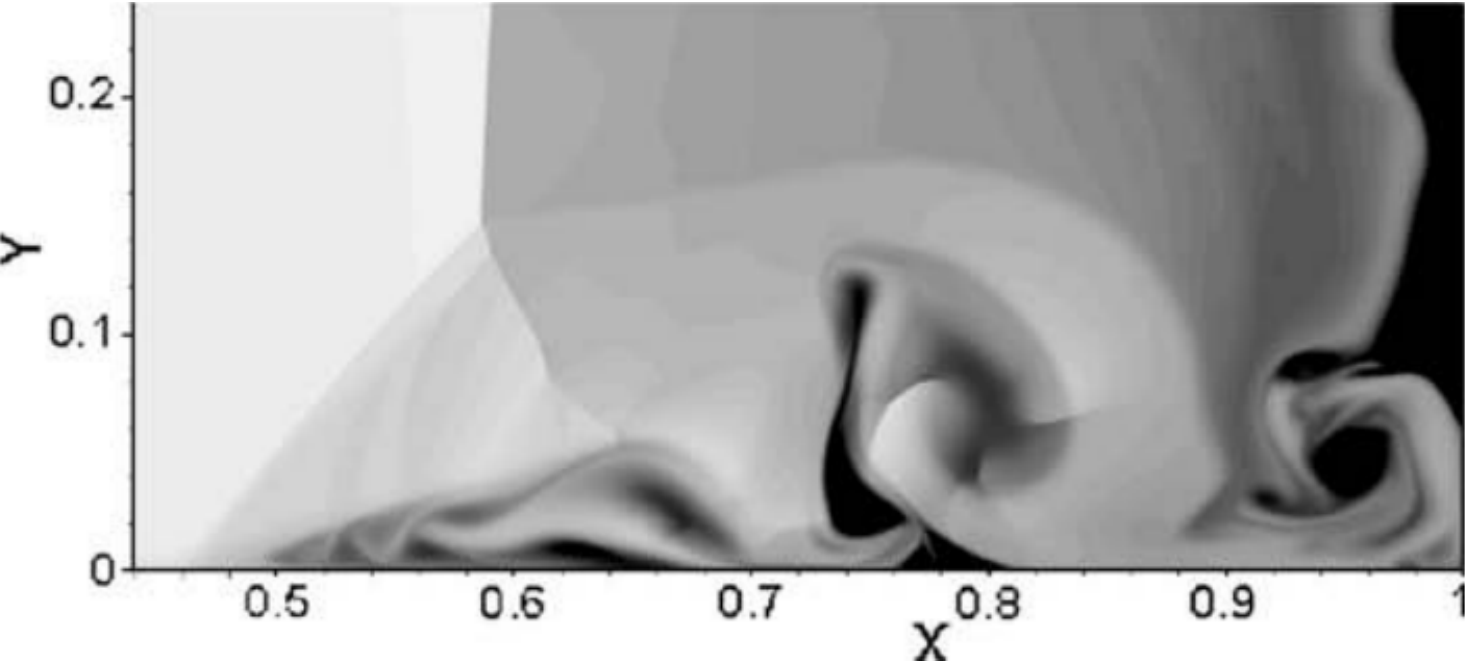}
    \end{subfigure}
    \begin{subfigure}{0.49\textwidth}
    \vspace{2mm}
    \includegraphics[width=0.98\textwidth, valign=t]{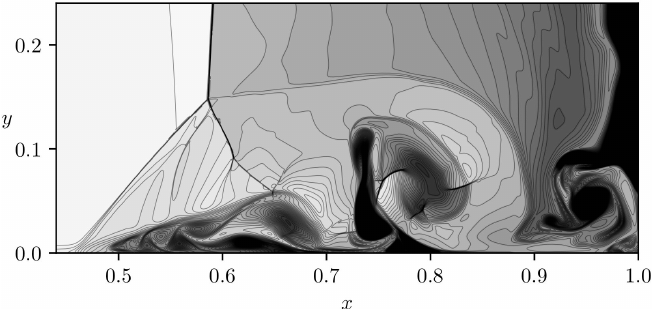}
    \end{subfigure}
   \end{subfigure}
   \begin{subfigure}{\linewidth}	
   \begin{subfigure}{0.49\textwidth}    \includegraphics[width=0.965\textwidth, valign=t]{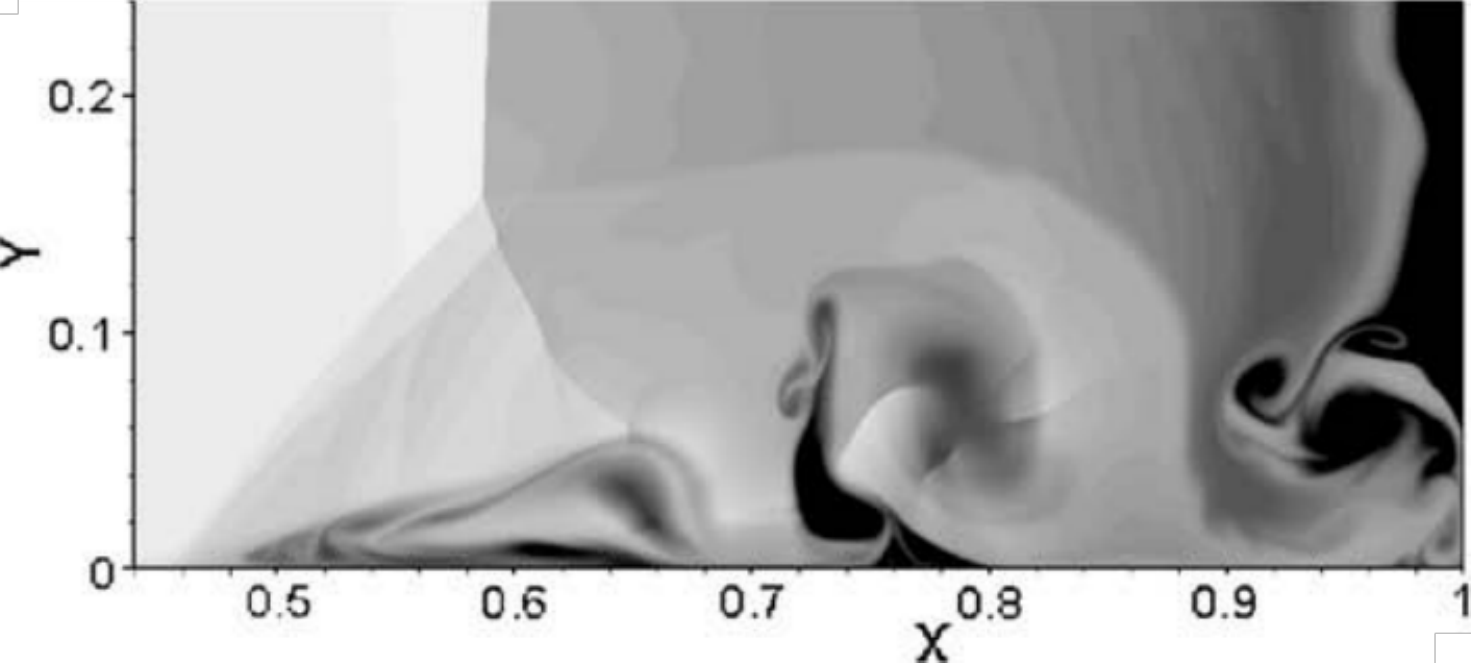}
    \end{subfigure}
    \begin{subfigure}{0.49\textwidth}
    \vspace{2mm}
    \includegraphics[width=0.98\textwidth, valign=t]{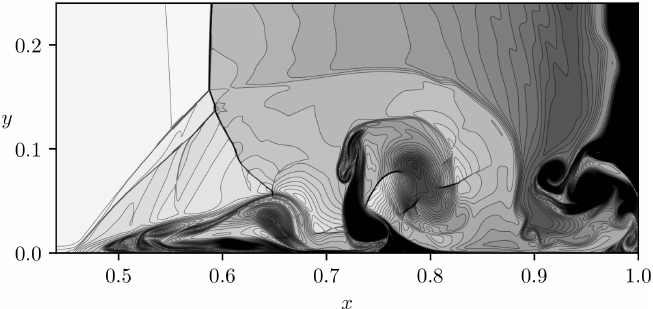}
    \end{subfigure}
   \end{subfigure}
   \caption{Contours of temperature (41 contour levels between 0.4 and 1.2) obtained at $t=1$ for several Reynolds numbers. Left: reference from~\cite{Daru2009} (OSMP7), right: present fourth-order scheme. From top to bottom: $\mathrm{Re}=200$, $\mathrm{Re}=500$, $\mathrm{Re}=750$ and $\mathrm{Re}=1000$.}
   \label{fig:ShockBoundary_t1}
\end{figure}

More quantitative results are shown in Fig.~\ref{fig:ShockBoundary_pos}, where the horizontal and vertical positions of the triple point are plotted as a function of time. As in~\cite{Daru2009}, the horizontal velocity is constant and nearly independent of the Reynolds number. On the contrary, the vertical position of the triple point strongly depends on the Reynolds number. The results obtained with the fourth-order kinetic scheme are in any case very close to that obtained with the OSMP7 scheme of \cite{Daru2009}. Notably, we recover the fact that, at $t\approx 0.53$, all the curves cross each other.

\begin{figure}[h!]
    \centering
    \begin{subfigure}{0.48\textwidth}
    \centering
    \includegraphics{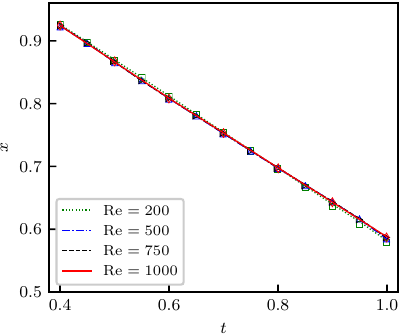}
    \end{subfigure}
    \begin{subfigure}{0.48\textwidth}
    \centering
    \includegraphics{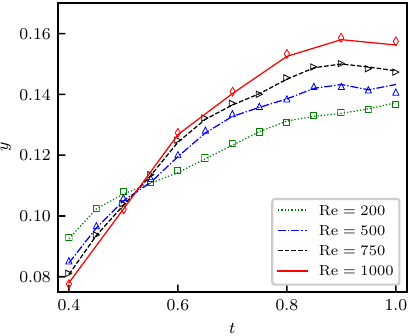}
    \end{subfigure}
    \caption{Horizontal (left) and vertical (right) coordinates of the triple point versus time. Symbols display the reference results from~\cite{Daru2009} at different Reynolds numbers.}
    \label{fig:ShockBoundary_pos} 
\end{figure}

A last important remark regards the ability of the proposed schemes to deal with discontinuities such as the shock waves encountered in the present test case. The fourth-order method investigated here has not been built in order to satisfy total variation diminushing (TVD) properties. Therefore, oscillations are likely to appear close to any kind of discontinuity. In all the results presented in this section, oscillations are indeed observed, but their amplitude is, fortunately, considerably damped thanks to the fluid viscosity and thermal diffusion. In these cases, they never generate local negative values of the pressure, density or temperature, which would lead to non-physical behaviors. In order to make this method robust in more delicate situations involving greater oscillations (\textit{e.g.} at higher Reynolds numbers or for coarser meshes), a non-linear stabilization strategy based on the MOOD technique~\cite{Vilar} can be employed, such as what was recently proposed in~\cite{Torlo, Abgrall2023}.

\section{Conclusion and perspectives}

This work introduced new kinetic methods to approximate linear and non-linear transport-diffusion systems like the Navier-Stokes equations for fluid dynamics. In~\cite{wissocq2023Kinetic}, the foundations of the proposed method have been laid for one-dimensional cases. They are based on two new ideas:
\begin{itemize}
	\item Compared to previous work dedicated to the investigation of kinetic methods in the diffusive limit~\cite{Jin, Jin1998, Jin2000, Naldi2000, Peng2020}, diffusion is not targeted in the strict limit $\varepsilon \rightarrow 0$ but for arbitrary low, non-vanishing, Knudsen numbers.
	\item The free parameters necessary to match a desired diffusion are introduced thanks to a specifically designed collision matrix, thereby establishing a coupling between the conserved variables within the relaxation term.
\end{itemize}
The present work has extended this framework to the general multi-dimensional case. An additional difficulty arises from the fact that the diffusion matrix is not, in general, block diagonal. For the Navier-Stokes equations, this is due to the existence of shear viscosity: the fluid motion in one direction generates a diffusive flux in the transverse direction, as evidenced for example in a Couette flow. Therefore, it is not possible to extend the idea of~\cite{wissocq2023Kinetic} with a simple dimensional splitting because directions have to be coupled with each other. In this work, this is achieved by reformulating the collision process in the basis of moments, inspired by the MRT models in the LBM community~\cite{DHumieres1994, Lallemand2000, DHumieres2002}. A major difference with standard MRT model, however, is that we do not assume the collision matrix to be diagonal in a given basis of moments. Instead, we make no assumption on the block matrix in the sub-space of first-order moments, and show, thanks to a Chapman-Enskog expansion, that a unique form of this sub-matrix can be found in order to match the transport-diffusion problem at first-order in Knudsen number. This enables the coupling of fluxes in all directions within the relaxation process, thus restoring the non-diagonal characteristic of the diffusion matrix. 

Furthermore, the reformulation of the collision matrix in the moments basis allows us to systematically ensure the conservation of the zeroth-order moments, as well as to explicitly control the behavior of additional high-order moments. In this regard, a regularization procedure, similar to what is sometimes done in the LBM community~\cite{Skordos1993, Ladd2001, Latt2006, Malaspinas2015, Coreixas2017}, can be adapted to the kinetic framework. Noticing that high-order moments are not involved in the Chapman-Enskog expansion at first-order in $\varepsilon$, they can be arbitrarily set to zero without affecting the recovery of the target convection-diffusion system. This strategy allows us to reinterpret our kinetic models as Jin-Xin ones, involving $d+1$ variables only, thus increasing the efficiency and reducing the memory cost of our method.

This strategy has been validated for one-dimensional cases in~\cite{wissocq2023Kinetic}, and for several two-dimensional cases in the present paper: linear advection-diffusion in two dimensions, two kinds of Couette flows and shock-boundary layer interactions at various Reynolds numbers. All the testcases confirm the existence of a consistency error scaling as $\mathcal{O}(\varepsilon^2)$, which can be arbitrarily reduced by increasing the kinetic speeds. In a sense, the issue reported in the diffusion limit of standard kinetic methods is still present: an \emph{exact} matching of the convection-diffusion problem could only be achieved in the limit of \emph{infinite} kinetic speeds, and the time step would tend to zero in this limit. However, it is important to note that in our model, the requirement for a small time step is linked to an accuracy condition, rather than a stability constraint as in previous models. By releasing the constraint of exact resolution of the advection-diffusion problem, thus accepting a consistency error, it is possible to ensure stability conditions with affordable time steps. This idea is inspired by the fact that Navier-Stokes is only an approximation of Boltzmann for small but nonzero Knudsen numbers. It is validated by the excellent results obtained with the shock-boundary layer interactions, where the time step is only constrained by the sub-characteristic condition of a four-wave kinetic model.

Further work will aim at validating these methods on more complex applications, potentially involving other types of transport-diffusion equations, and increasing the computational efficiency of the numerical method. This may be achieved by a GPU implementation of the scheme, which can speed up computations of matrix products and local inversions. An extension of this method to non-structured grids may also be the purpose of future work.

\section*{Acknowledgements}
GW has been funded by SNFS grants \# 
200020\_204917 ``Structure preserving and fast methods for hyperbolic systems of conservation laws'' and 
FZEB-0-166980.

\appendix

\section{Chapman-Enskog expansion of the BGK system}
\label{app:CE_BGK}

We start by left-multiplying \eqref{eq:kinetic_with_epsilon} by $\P$, leading to
\begin{align}
	\label{eq:app_CE_BGK_1}
	\dpar{\bbu^\varepsilon}{t} + \sum_{i=1}^d \dpar{}{x_i} \left( \P \Lambda_i \bF \right) = \mathbf{0},
\end{align}
where we note $\bbu^\varepsilon = \P \bF$ and where we used the fact that $\P \M (\bbu^\varepsilon) = \bbu^\varepsilon$. This is a conservation equation of $\bbu^\varepsilon$ involving fluxes $\P \Lambda_i \bF$. The purpose of the Chapman-Enskog expansion is to provide an approximation of this flux for small values of $\varepsilon$. To do so, we rewrite \eqref{eq:kinetic_with_epsilon} as
\begin{align}
	\bF = \M(\bbu^\varepsilon) - \varepsilon \, \frac{\ell}{||\Lambda||} \left( \dpar{\bF}{t} + \sum_{j=1}^{d} \Lambda_j \dpar{\bF}{x_j} \right).
\end{align}
In particular, we see that $\bF = \M(\bbu^\varepsilon) + \mathcal{O}(\varepsilon)$. Hence,
\begin{align}
	\bF = \M(\bbu^\varepsilon) - \varepsilon \, \frac{\ell}{||\Lambda||} \left( \dpar{\M(\bbu^\varepsilon)}{t} + \sum_{j=1}^{d} \Lambda_j \dpar{\M(\bbu^\varepsilon)}{x_j} \right) + \mathcal{O}(\varepsilon^2).
\end{align}
Then left-multiplying by $\P \Lambda_i$,
\begin{align}
	\P \Lambda_i \bF = \bbf_i(\bbu^\varepsilon) - \varepsilon \frac{\ell}{||\Lambda||} \left( \dpar{\bbf_i(\bbu^\varepsilon)}{t} + \sum_{j=1}^d \dpar{}{x_j} \left( \P \Lambda_i \Lambda_j \M(\bbu^{\varepsilon}) \right) \right) + \mathcal{O}(\varepsilon^2),
\end{align}
where we used the fact that $\P \Lambda_i \M(\bbu^\varepsilon) = \bbf_i(\bbu^\varepsilon)$. The time and space derivatives can be disposed thanks to chain rules,
\begin{align}
	& \dpar{\bbf_i(\bbu^\varepsilon)}{t} = \bbf_i'(\bbu^\varepsilon) \dpar{\bbu^\varepsilon}{t} = -\bbf_i'(\bbu^\varepsilon) \sum_{j=1}^d \dpar{\bbf_j(\bbu^\varepsilon)}{x_j} + \mathcal{O}(\varepsilon) = -\sum_{j=1}^d \bbf_i'(\bbu^\varepsilon) \bbf_j'(\bbu^\varepsilon) \dpar{\bbu^\varepsilon}{x_j} + \mathcal{O}(\varepsilon), \\
	& \dpar{}{x_j} \left( \P \Lambda_i \Lambda_j \M(\bbu^\varepsilon) \right) = \P \Lambda_i \Lambda_j \M'(\bbu^\varepsilon) \dpar{\bbu^\varepsilon}{x_j}.
\end{align}
This leads to
\begin{align}
	\P \Lambda_i \bF = \bbf_i(\bbu^\varepsilon) - \varepsilon \frac{\ell}{||\Lambda||} \sum_{j=1}^d \left[ \P \Lambda_i \Lambda_j \M'(\bbu^\varepsilon) - \bbf_i'(\bbu^\varepsilon) \bbf_j'(\bbu^\varepsilon) \right] \dpar{\bbu^\varepsilon}{x_j} + \mathcal{O}(\varepsilon^2),
\end{align}
which, when injected in \eqref{eq:app_CE_BGK_1}, leads to \eqref{eq:CE_BGK}.

\end{document}